\documentclass[12pt,a4paper]{article}
\usepackage[a4paper, top=2.5cm, bottom=2.5cm, left=2.5cm, right=2.5cm]{geometry}
\usepackage{amsmath,amssymb,amsthm,bbm,mathtools,mathrsfs,graphicx,pstricks-add,caption,blkarray,mathdots,enumerate,tikz,tikz-3dplot,enumitem,verbatim,listings,xcolor,cite}

\usepackage{hyperref,authblk}
\usepackage{todonotes}
\usepackage{dsfont,mathabx}
\theoremstyle{plain}
\newtheorem{thm}{Theorem}
\newtheorem*{thm*}{Theorem}
\newtheorem{cor}[thm]{Corollary}
\newtheorem{lem}[thm]{Lemma}
\theoremstyle{definition}
\newtheorem{definition}[thm]{Definition}
\newtheorem{remark}[thm]{Remark}
\newtheorem{ex}[thm]{Example}

\numberwithin{equation}{section}

\DeclareMathOperator{\Pb}{\mathbb{P}}
\DeclareMathOperator{\Ex}{\mathbb{E}}
\newcommand{\disp}{\displaystyle}

\DeclareMathOperator{\I}{\mathbb{I}}

\DeclareMathOperator{\amax}{argmax}
\DeclareMathOperator{\amin}{argmin}
\DeclareMathOperator{\var}{Var}

\definecolor{dgreen}{rgb}{0,0.6,0}
\definecolor{codegray}{gray}{0.9}

\newcommand{\set}[1]{\left\{#1\right\}}

\newcommand{\paren}[1]{\left(#1\right)}
\newcommand{\brac}[1]{\left[#1\right]}

\newcommand{\eps}{\varepsilon}

\usetikzlibrary{calc,patterns}
\usepackage{titlesec}
\titleformat{\chapter}[display]
  {\Huge\bfseries}
  {}
  {0pt}
  {\thechapter.\ }

\titleformat{name=\chapter,numberless}[display]
  {\Huge\bfseries}
  {}
  {0pt}
  {}

\usepackage[normalem]{ulem} 
\allowdisplaybreaks
\newcommand{\edit}[1]{{#1}}

\begin{document}
\title{Tau-Rho Equality and Other Dependence Measures of a Subclass of Factorizable Copulas
}
\author[1]{Noppawit Yanpaisan}
\author[1]{Tippawan Santiwipanont}
\author[2]{Matthias Scherer}
\author[1]{Songkiat Sumetkijakan\footnote{Corresponding author (songkiat.s@chula.ac.th)}}
\affil[1]{\small Chulalongkorn University, Bangkok 10330, Thailand}
\affil[2]{\small Technical University of Munich, Garching--Hochbr\"{u}ck 85748, Germany}

\date{}

\maketitle
\begin{abstract}
Kendall's tau and Spearman's rho, two widely used dependence measures in statistics and risk management, are often treated as interchangeable, yet can disagree sharply: Schreyer et al.~(2017) established the exact region of attainable $(\tau,\rho)$ pairs. 
We study the complementary question of equality, namely, identifying nontrivial families of copulas $C$ satisfying $\tau_C=\rho_C$.
We prove that this equality holds for every factorizable copula of the form $C_{e,\alpha}\ast C_{\beta,e}$, where $\alpha$ and $\beta$ are piecewise linear monotonic surjections (PLMS).
For these PLMS-generated copulas, we study further dependence measures, including Chatterjee's rank correlation coefficient and tail dependence coefficients, revealing some useful algebraic formulas and unexpected phenomena. 
In particular, Chatterjee's coefficient can exhibit extreme asymmetry.
%
\\[10pt]
\end{abstract}
\section{Introduction: Copulas with $\tau_C=\rho_C$}
Bivariate dependence measures, particularly Kendall’s tau and Spearman’s rho, have a rich history in probability theory and statistics, see, \edit{e.g.,\ \cite{Daniels50,prin15,Durbin51,tau-rho07,FSW26,KT20,QRM15,cop06,Renyi59,Rus13,Tau&Rho2017}}. The literature on copulas provides  explicit formulas $\tau_C$ and $\rho_C$ for most copula families $C$, see, \edit{e.g.,\ \cite{FSW26,KT20,cop06}}. Recently, a lot of interest was sparked by questions like: “\emph{How different can $\tau_C$ and $\rho_C$ be?}” and “\emph{Which class of copulas maximizes $\vert \tau_C-\rho_C\vert$?},” \edit{see \cite{Daniels50,Durbin51,tau-rho07,KT20,Tau&Rho2017}}.
\par 
Reflecting on the different principles behind Kendall’s $\tau$\footnote{Classically defined via the probability for concordant and discordant pairs.} and Spearman’s $\rho$\footnote{Stemming from Pearson’s correlation after transforming the marginal laws to uniform $\mathcal{U}[0,1]$.} and recalling the analytical formulas for computing these concordance measures, see, \edit{e.g.,\ \cite{tau-rho07,FSW26,cop06,Tau&Rho2017}}, we argue that one should actually not be surprised to find $\tau_C\neq \rho_C$ for most $C$. Motivated by this observation, our investigation begins with the complementary question to the one above:
\emph{Are there non-trivial copulas satisfying $\tau_C=\rho_C$?}
\par
Besides the obvious candidates independence $\Pi(u,v)=u\, v$, comonotonicity $M(u,v)=\min\set{u,v}$, and countermonotonicity $W(u,v)=\max\set{u+v-1,0}$,
where both dependence measures are $0$, $1$, and $-1$, respectively, we found it surprisingly difficult to identify further candidates. 
Having in mind that various optimization problems involving copulas are solved by copulas supported on a set of Lebesgue-measure zero \cite{2D-1D11,dist15,RiskAgg2026}, we considered the family of implicit dependence copulas as a promising candidate and were indeed able to identify a \edit{subfamily consisting of special factorizable copulas} that 
 satisfy $\tau_C=\rho_C$. 
\par 
This paper is organized as follows: Firstly, necessary definitions and tools for our investigation are collected in Section~\ref{background}.
In Section~\ref{sec:PLMS}, we investigate in quite some detail properties of the family of copulas constructed from piecewise linear monotonic surjections (PLMS). Most notably, for members of this family we find $\tau_C=\rho_C$. We extend our journey to factorizable and implicit dependence copulas, respectively. Subsequently, more results related to PLMS functions and their corresponding copulas are given in Section~\ref{sec:PLMS-prop}, which includes Chatterjee's rank correlation and tail dependence. For the former, which is asymmetric by nature, we showed that a copula can be constructed such that for $(U,V)\sim C$ one has $\xi(U,V)=1$, while $\xi(V,U)<\eps$ for arbitrarily small $\eps>0$; which is somewhat counterintuitive.

\section{Preliminaries and notation}\label{background}
First of all, we denote the unit interval by $\I:=[0,1]$, use $\lambda$ as the Lebesgue measure on $\I$, let $\mathscr{B}$ be the Borel $\sigma$-algebra on $\I$, $\mathscr{C}$ the class of two-dimensional copulas, and $\mathscr{T}$ the class of measure-preserving transformations on $\I$. In other words, $\alpha\in\mathscr{T}$ if $\lambda\paren{\alpha^{-1}(B)}=\lambda(B)$ for any $B\in\mathscr{B}$, see \cite{reapro04}.
Note that every copula $C\in\mathscr{C}$ can be represented as $C\paren{x,y}=C_{\alpha,\beta}\paren{x,y}:=\lambda\paren{\alpha^{-1}[0,x]\cap\beta^{-1}[0,y]}$ for some $\alpha,\beta\in\mathscr{T}$, but the uniqueness of the pair $\paren{\alpha,\beta}$ does not hold, see \cite{mea11,Vit1996}.
\par 
Algebraically, the class $\mathscr{C}$ is not a vector space but a convex space. A well-known binary operation on $\mathscr{C}$ is the (usual) \emph{Markov product} $\ast$, which is defined for $C,D\in\mathscr{C}$ as $C\ast D(x,y):=\int_0^1\partial_2C\paren{x,t}\partial_1D\paren{t,y}dt$, see \cite{prin15}. 
Recall that $\partial_i$ denotes the partial derivative with respect to the $i^{\text{th}}$ variable.
In addition, $\paren{\mathscr{C},\ast}$ is a monoid with the identity $M\colon(x,y)\mapsto\min\set{x,y}$, the null element $\Pi\colon(x,y)\mapsto xy$, and one-sided invertible elements of the form $C_{e,\alpha}$ or $C_{\alpha,e}$ for some $\alpha\in\mathscr{T}$; $e$ being the identity on $\I$.
 
\par 
The Markov product can be generalized to $C\ast_{\mathcal{A}}D$, where $\mathcal{A}=\paren{A_t}_{t\in\I}$ is a parametric class of copulas and each $A_t$ may be called a \emph{joining copula}, by changing the product of partial derivatives into $A_t\paren{\partial_2C\paren{x,t},\partial_1D\paren{t,y}}$, see \cite{product21,rem07,prin15}.
Depending on the integrand, $C\ast_{\mathcal{A}}D$ is called a \emph{generalized Markov product} of $C$ and $D$ whenever it is well-defined.
\footnote{The generalized Markov product is found in the compatibility problem between three random variables with two given bivariate marginals \cite{boundstrivariate08,rem07}.
    In other words, if $X,Y,Z$ are random variables with $\paren{X,Y}\sim C$ and $\paren{Y,Z}\sim D$, then $\paren{X,Z}\sim C\ast_{\mathcal{A}}D$ for some $\mathcal{A}\subseteq\mathscr{C}$.}
\par 
\emph{Dependence measures} can support different tasks, e.g.\ fitting the parameters of a copula, running hypothesis tests for independence, and obviously measuring type and strength of dependence. Many dependence measures, among others Kendall's tau and Spearman's rho, fall under the umbrella of so-called \emph{concordance measures}. \par 
In terms of a statistical interpretation, suppose that $S:=\set{\paren{x_1,y_1},\dots,\paren{x_n,y_n}}$ is a random sample of $n$ observations from a continuous random vector $\paren{X,Y}$, where $X$ and $Y$ are distributed according to the distribution functions $F$ and $G$, respectively.
The empirical version of \emph{Kendall's tau} of the sample $S$ is then defined by $(c-d)/(c+d)$, where $$c:=\left|\set{\paren{i,j}:(x_i-x_j)(y_i-y_j)>0}\right|\text{ and }d:=\left|\set{\paren{i,j}:(x_i-x_j)(y_i-y_j)<0}\right|$$ are the numbers of \emph{concordance} and \emph{discordance} pairs, respectively.
In addition, define a transformed random sample by $S':=\set{\paren{F(x_1),G(y_1)},\dots,\paren{F(x_n),G(y_n)}}$.
Then \emph{Spearman's rho} for the random sample $S$ is defined via the best linear approximation to $S'$ in least squares sense, i.e., the sample correlation among the transformed observations, see \cite{cop06,Wasserman04}.\par 
In terms of a probabilistic definition, if a copula of the random vector $\paren{X,Y}$ is $C$, then Kendall's tau and Spearman's rho of $\paren{X,Y}$ can be written in terms of $C$:
\begin{align}
\tau_{X,Y}:=\tau_C&=4\iint_{\I^2}C\paren{x,y}dC\paren{x,y}-1=1-4\iint_{\I^2}\frac{\partial C}{\partial x}\frac{\partial C}{\partial y}dxdy\quad\text{and}\label{tau-def}\\
\rho_{X,Y}:=\rho_C&=12\iint_{\I^2}C\paren{x,y}dxdy-3.\label{rho-def}
\end{align}
Relations between $\tau_C$ and $\rho_C$, as well as the region of points $\paren{\tau_C,\rho_C}$ for $C\in\mathscr{C}$, are investigated in the literature, e.g., \cite{cop06,Tau&Rho2017}. 
\par
Apart from concordance measures, there exist other measures of dependence which are suitable to detect `how far from independence' the relationship between the given random variables is. One of these quantities is introduced by Chatterjee \cite{Chat21}, namely the \emph{Chatterjee's rank correlation}. For the empirical version, let $S$ be the same setup as previously mentioned. Instead of searching the best linear model for the corresponding transformed sample $S'$, as done with Spearman's rho, the \emph{statistical rank} is used to formulate the correlation by rearranging $S$ to $S'':=\set{\paren{x_{(1)},y_{(1)}},\dots,\paren{x_{(n)},y_{(n)}}}$ where $x_{(1)}\le x_{(2)}\le\dots\le x_{(n)}$, and denoting $r_i$ as the number of $j$'s such that $y_{(j)}\le y_{(i)}$. Then, it is shown in \cite{Chat21} that $\xi_n\paren{X,Y}$ is a consistent estimator of an asymmetric measure of dependence $\xi\paren{X,Y}$, with
$$\xi_n\paren{X,Y}:=1-\frac{3\sum_{i=1}^{n-1}\left|r_{i+1}-r_i\right|}{n^2-1}\quad\text{and}\quad\xi(X,Y)=\dfrac{\int\var(\Ex[\mathbbm{1}_{Y\ge t}\mid X])dG(t)}{\int\var(\mathbbm{1}_{Y\ge t})dG(t)},$$
the latter being the probabilistic counterpart introduced in \cite{copbase13}. The formula of $\xi_n\paren{X,Y}$ can be modified if some $x_{(i)}$'s or $y_{(i)}$'s tie.
Chatterjee's rank correlation plays an important role as a test statistic for independence testing.
It is very efficient in terms of computational cost, but it seems to have a higher Type II error rate than competitors, especially when the data is smooth and non-oscillating, see \cite{Chat21}.
Moreover, $\xi\paren{X,Y}$ can actually be written in the copula-based form $\xi_C$ \cite{copbase13,Renyi59}, where $C$ is the copula of $\paren{X,Y}$, as 
\begin{equation}\label{xi-copbase}
\xi_C=6\iint_{\I^2}\paren{\partial_1C\paren{x,y}}^2dxdy-2.  
\end{equation}\par
Consider again the vector $\paren{X,Y}$ with $X\sim F$ and $Y\sim G$ and copula $C$. When joint tail events are relevant, some determining quantities are the \emph{tail dependence coefficients}. With $H^{\leftarrow}$ denoting the \emph{quantile function} of a distribution function $H$, the \emph{lower} and \emph{upper tail dependence coefficients} of $\paren{X,Y}$ are defined by
$$\lambda_{X,Y}^L:=\lim\limits_{t\searrow 0}\Pb\paren{Y\le G^{\leftarrow}(t)\mid X\le F^{\leftarrow}(t)}\quad\text{and}\quad \lambda_{X,Y}^U:=\lim\limits_{t\nearrow 1}\Pb\paren{Y>G^{\leftarrow}(t)\mid X>F^{\leftarrow}(t)},$$
respectively \cite{cop06},
which can be expressed in terms of the corresponding copula $C$ by
\begin{equation}\label{tail-dep}
\lambda_C^L=\lim\limits_{t\searrow 0}\frac{C\paren{t,t}}{t}\quad\text{and}\quad\lambda_C^U=\lim\limits_{t\nearrow 1}\frac{1-2t+C\paren{t,t}}{1-t}.
\end{equation}

\section{Copulas from piecewise linear monotonic surjections (PLMS)}\label{sec:PLMS}
While the statistical community typically relies on copulas having a classical density, enabling methods such as maximum-likelihood estimation, copulas with singular components or the ones supported only on a set of Lebesgue measure zero --~the latter, as a stronger property, implying the former, see \cite{note-sing-13}~-- play an important role in many fields of probability. Often, they solve optimization problems over the Fr\'echet class or provide the solution to some mass-transportation problem. Moreover, they often serve as counterexamples or provide cases that require specific treatment. Our question of investigation also falls into this category.   
\subsection{Complete dependence copulas}
Among the simplest copulas supported on a set of Lebesgue measure zero are the ones constructed from the graph of a piecewise linear measure-preserving transformation. If each line gives a surjective affine map, the resulting copula will have interesting properties with respect to its concordance measures, as shown in the following example.
\begin{ex}\label{sec3-ex1}
        Let $\theta\in(0,1)$ and divide $[0,1]$ into the subintervals $[0,\theta)$ and $[\theta,1]$.
        There are four possible choices of copulas concentrated on two line segments, such that the image of each open subinterval is $\paren{0,1}$, as shown in the first row of Figure~\ref{ex1}.
        The concordance measures  Kendall's tau and Spearman's rho are computed according to their probabilistic definitions, and their values are stated in the second row of Figure~\ref{ex1}; we observe that $\tau_C=\rho_C$ for every $C$ and thus only report $\tau_C$.  Furthermore, $\tau_{C_{\theta}^{\oplus,\oplus}}=-\tau_{C_{\theta}^{\ominus,\ominus}}$ and $\tau_{C_{\theta}^{\oplus,\ominus}}=-\tau_{C_{\theta}^{\ominus,\oplus}}$.
        The general formulas of Kendall's tau and Spearman's rho of $C_{e,\alpha}$ with $\alpha\in\mathscr{T}_{\text{PLMS}}$ are given in Theorem~\ref{tau-rho-1}.
    \begin{figure}
        \centering
        \begin{tikzpicture}[scale=1.5]
            \draw ($(0,0)+(0,2)$) rectangle ($(1,1)+(0,2)$);
            \draw[dashed] ($(0.4,1)+(0,2)$)--($(0.4,0)+(0,2)$) node[below]{$\theta$};
            \draw[very thick]($(0,0)+(0,2)$)--($(0.4,1)+(0,2)$);
            \draw[very thick]($(0.4,0)+(0,2)$)--($(1,1)+(0,2)$);
            \node[above] at ($(0.5,1)+(0,2)$){$C_{\theta}^{\oplus,\oplus}$};
            \draw ($(0,0)+(2,2)$) rectangle ($(1,1)+(2,2)$);
            \draw[dashed] ($(0.4,1)+(2,2)$)--($(0.4,0)+(2,2)$) node[below]{$\theta$};
            \draw[very thick]($(0,0)+(2,2)$)--($(0.4,1)+(2,2)$)--($(1,0)+(2,2)$);
            \node[above] at ($(0.5,1)+(2,2)$){$C_{\theta}^{\oplus,\ominus}$};
            \draw ($(0,0)+(4,2)$) rectangle ($(1,1)+(4,2)$);
            \draw[dashed] ($(0.4,1)+(4,2)$)--($(0.4,0)+(4,2)$) node[below]{$\theta$};
            \draw[very thick]($(0,1)+(4,2)$)--($(0.4,0)+(4,2)$)--($(1,1)+(4,2)$);
            \node[above] at ($(0.5,1)+(4,2)$){$C_{\theta}^{\ominus,\oplus}$};
            \draw ($(0,0)+(6,2)$) rectangle ($(1,1)+(6,2)$);
            \draw[dashed] ($(0.4,1)+(6,2)$)--($(0.4,0)+(6,2)$) node[below]{$\theta$};
            \draw[very thick]($(0,1)+(6,2)$)--($(0.4,0)+(6,2)$);
            \draw[very thick]($(0.4,1)+(6,2)$)--($(1,0)+(6,2)$);
            \node[above] at ($(0.5,1)+(6,2)$){$C_{\theta}^{\ominus,\ominus}$};
            \node[left] at (0,0){\footnotesize$0$};
            \draw[dashed](0,1)--(1,1)--(1,0) node[below] {\footnotesize$1$};
            \draw (0,1)--(0,0)--(1,0) node[right] {\footnotesize$\theta$};
            \node[left] at (0,1){\footnotesize$1$};
            \draw[dashed](0,0.5)--(0.5,0.5)--(0.5,0) node[below] {\footnotesize$\frac{1}{2}$};
            \node[left] at (0,0.5){\footnotesize$\frac{1}{2}$};
            \node[above] at (0.5,1){\footnotesize$\tau_{C_{\theta}^{\oplus,\oplus}}=1-2\theta(1-\theta)$};
            \draw[very thick,domain=0:1,smooth] plot(\x,{1-2*\x*(1-\x)});
            \node[left] at ($(0,0)+(2,0)$){\footnotesize$0$};
            \node[below] at ($(1,0)+(2,0)$){\footnotesize$1$};
            \draw[dashed]($(0,1)+(2,0)$)--($(1,1)+(2,0)$)--($(1,-1)+(2,0)$)--($(0,-1)+(2,0)$);
            \draw ($(0,0)+(2,0)$)--($(1,0)+(2,0)$) node[right] {\footnotesize$\theta$};
            \draw ($(0,1)+(2,0)$)--($(0,-1)+(2,0)$) node[left]{\footnotesize$-1$};
            \node[left] at ($(0,1)+(2,0)$){\footnotesize$1$};
            \node[below] at ($(0.5,0)+(2,0)$){\footnotesize$\frac{1}{2}$};
            \node[above] at ($(0.5,1)+(2,0)$){\footnotesize$\tau_{C_{\theta}^{\oplus,\ominus}}=2\theta-1$};
            \draw[very thick,domain=0:1,smooth] plot({\x+2},{2*\x-1});
            \node[left] at ($(0,0)+(4,0)$){\footnotesize$0$};
            \node[below] at ($(1,0)+(4,0)$){\footnotesize$1$};
            \draw[dashed]($(0,1)+(4,0)$)--($(1,1)+(4,0)$)--($(1,-1)+(4,0)$)--($(0,-1)+(4,0)$);
            \draw ($(0,0)+(4,0)$)--($(1,0)+(4,0)$) node[right] {\footnotesize$\theta$};
            \draw ($(0,1)+(4,0)$)--($(0,-1)+(4,0)$) node[left]{\footnotesize$-1$};
            \node[left] at ($(0,1)+(4,0)$){\footnotesize$1$};
            \node[below] at ($(0.5,0)+(4,0)$){\footnotesize$\frac{1}{2}$};
            \node[above] at ($(0.5,1)+(4,0)$){\footnotesize$\tau_{C_{\theta}^{\ominus,\oplus}}=1-2\theta$};
            \draw[very thick,domain=0:1,smooth] plot({\x+4},{1-2*\x});
            \node[left] at ($(0,1)+(6,0)$){\footnotesize$0$};
            \node[left] at ($(0,0)+(6,0)$){\footnotesize$-1$};
            \draw[dashed]($(0,0)+(6,0)$)--($(1,0)+(6,0)$)--($(1,1)+(6,0)$) node[below] {\footnotesize$1$};
            \draw ($(0,0)+(6,0)$)--($(0,1)+(6,0)$)--($(1,1)+(6,0)$) node[right] {\footnotesize$\theta$};
            \node[left] at ($(0,1)+(6,0)$){\footnotesize$0$};
            \draw[dashed]($(0,0.5)+(6,0)$)--($(0.5,0.5)+(6,0)$)--($(0.5,1)+(6,0)$) node[below] {\footnotesize$\frac{1}{2}$};
            \node[left] at ($(0,0.5)+(6,0)$){\footnotesize$-\frac{1}{2}$};
            \node[above] at ($(0.5,1)+(6,0)$){\footnotesize$\tau_{C_{\theta}^{\ominus,\ominus}}=2\theta(1-\theta)-1$};
            \draw[very thick,domain=0:1,smooth] plot({\x+6},{2*\x*(1-\x)-1+1});
        \end{tikzpicture}
        \caption{Support of $C_{e,\alpha}$ and the graph of the resulting Kendall's tau (= Spearman's rho) in terms of $\theta$, where $\alpha\in\mathscr{T}_{\text{PLMS}}$ whose support consists of two line segments.}
        \label{ex1}
    \end{figure}
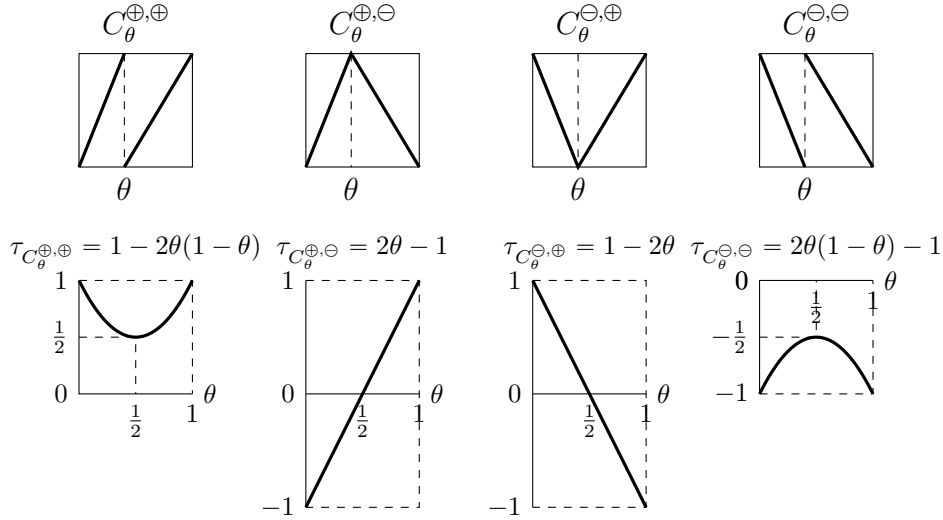
    \end{ex}
Example~\ref{sec3-ex1} is generalized in the following definition.
\begin{definition}[\!\!\cite{RiskAgg2026}]
    A function $\alpha\in\mathscr{T}$ is called a \emph{piecewise linear monotonic surjection} (PLMS), denoted by $\alpha\in\mathscr{T}_{\text{PLMS}}$, if there is a partition $0=a_0<a_1<\dots<a_{N-1}<a_N=1$, or $[a_i]_{i=0}^{N}$ in short, such that the graph of each $\alpha_i:=\alpha\mid_{\paren{a_{i-1},a_i}}$ is a diagonal of the rectangle $(a_{i-1},a_i)\times(0,1)$.
    More precisely, for any $x\in (a_{i-1},a_i)$:
    $$ \alpha_i(x)=\begin{cases}
        \dfrac{x-a_{i-1}}{a_i-a_{i-1}} & \text{if}\ \alpha_i\ \text{is increasing};\\
        \dfrac{a_i-x}{a_i-a_{i-1}} & \text{if}\ \alpha_i\ \text{is decreasing}.
    \end{cases}$$
\end{definition}
\begin{remark}
    Although there is no restriction on defining $\alpha\in\mathscr{T}_{\text{PLMS}}$ corresponding to the partition $[a_i]_{i=0}^{N}$ at every $x\in[a_i]_{i=0}^{N}$ for well-definedness, whenever $\alpha_k$, for $k=1,2,\dots,N$, is referred individually, we usually suppose the value of $\alpha_k$ at the endpoints for convenience and conciseness as $\alpha_k\paren{a_{k-1}}:=\lim\limits_{x\searrow a_{k-1}}\alpha_k(x)$ and $\alpha_k\paren{a_k}:=\lim\limits_{x\nearrow a_k}\alpha_k(x)$, which are either $0$ or $1$ (cannot be the same values simultaneously). 
\end{remark}
\begin{remark}
    When all $\alpha_k$, $k=1,2,\dots,N$, are increasing, the map $\alpha$ may be called the $[a_i]$-L\"{u}roth map. 
    These functions are used to define L\"{u}roth double shuffles, which induce collections of extreme points on the set of invariant copulas corresponding to the shuffles, see \cite{2-shuffle-Luroth22} for details.
\end{remark}
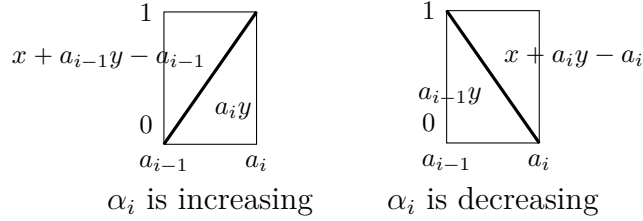
\begin{figure}
        \centering
        \begin{tikzpicture}[scale=1.75]
            \draw (0,0) rectangle (0.7,1);
            \draw[very thick](0,0)--(0.7,1);
            \node[below] at (0,0) {\footnotesize$a_{i-1}$};
            \node[below] at (0.7,0) {\footnotesize$a_{i}$};
            \node[anchor=south east] at (0,0){\footnotesize 0};
            \node[left] at (0,1){\footnotesize 1};
            \node[left] at (0.4,0.65){\footnotesize $x+a_{i-1}y-a_{i-1}$};
            \node[right] at (0.3,0.25){\footnotesize $a_iy$};
            \node[below] at (0.35,-0.25) {$\alpha_i$ is increasing};
        \end{tikzpicture} \hspace{1em}
        \begin{tikzpicture}[scale=1.75]
            \draw (0,0) rectangle (0.7,1);
            \draw[very thick](0,1)--(0.7,0);
            \node[below] at (0,0) {\footnotesize$a_{i-1}$};
            \node[below] at (0.7,0) {\footnotesize$a_{i}$};
            \node[anchor=south east] at (0,0){\footnotesize 0};
            \node[left] at (0,1){\footnotesize 1};
            \node[left] at (0.35,0.35){\footnotesize $a_{i-1}y$};
            \node[right] at (0.35,0.65){\footnotesize $x+a_iy-a_i$};
            \node[below] at (0.35,-0.25) {$\alpha_i$ is decreasing};
        \end{tikzpicture}
        \caption{Pointwise value of $C_{e,\alpha}$, where (left) $\alpha_i$ is increasing and (right) $\alpha_i$ is decreasing.}
        \label{alpha_i}
    \end{figure}
    Throughout this manuscript, measure-preserving transformations are assumed to be in $\mathscr{T}_{\text{PLMS}}$.
    Since a copula whose support is concentrated on $y=\alpha(x)$ can be written as $C_{e,\alpha}$, the explicit formula of the pointwise value of $C_{e,\alpha}$, for $x\in (a_{i-1},a_i)$, is
    \begin{equation}\label{PLMS-CDC}
        C_{e,\alpha}(x,y)=\begin{cases}
        \min\set{x+a_{i-1}y-a_{i-1},a_iy} & \text{if}\ \alpha_i\ \text{is increasing};\\
        \max\set{x+a_iy-a_i,a_{i-1}y} & \text{if}\ \alpha_i\ \text{is decreasing},
    \end{cases}
    \end{equation}
    as shown in Figure \ref{alpha_i}.
    In particular, $C_{e,\alpha}(a_i,y)=a_iy$ for every $i$.\par 
Observe that Example~\ref{sec3-ex1} suggests a collection of non-obvious candidates of copulas $C$ for which $\tau_C=\rho_C$. We shall see that this is a particular case of PLMS complete dependence copulas, where the result is true in general. Furthermore, the formula of $\tau_C$ (or $\rho_C$) can be expressed explicitly as follows.
\begin{thm}[$\tau=\rho$ for PLMS]\label{tau-rho-1}
    Let $\alpha\in\mathscr{T}_{\text{PLMS}}$ with a partition $[a_i]_{i=0}^{N}$. 
    Then
    \[\tau_{C_{e,\alpha}}=\rho_{C_{e,\alpha}}=2\sum_{i=1}^{N}w_i(a_i-a_{i-1})-1,\quad\text{where}\ w_i:=\underset{a_{i-1}\le x\le a_i}{\amax}\alpha_i(x).\]
\end{thm}
\begin{proof} Before verifying this claim, we introduce $\triangle(a_i):=a_i-a_{i-1}$ for notational simplicity.
    \begin{itemize}
        \item \emph{Kendall's tau}: To derive $\tau_{C_{e,\alpha}}$, note that 
    $$\tau_{C_{e,\alpha}}=4\Ex\big[\Ex[C(U,V)\mid U]\big]-1=4\Ex\big[C(U,\alpha(U))\big]-1=4\int_0^1C(u,\alpha(u))du-1.$$ The integral over each $\paren{a_{i-1},a_i}$ is treated in 2 cases.\\
    \textbf{Case 1:} $\alpha_i$ is increasing.
    Then \eqref{PLMS-CDC} implies $C(u,\alpha(u))=a_i\alpha_i(u)=\frac{w_i(u-a_{i-1})}{\triangle(a_i)}$ and
    \[4\int_{a_{i-1}}^{a_i}C(u,\alpha(u))du=4w_i\int_{a_{i-1}}^{a_i}\frac{u-a_{i-1}}{\triangle(a_i)}du=4w_i\triangle(a_i)\int_0^1z\,dz=2w_i\triangle(a_i). \tag{$z=\alpha_i(u)$}\]
    \textbf{Case 2:} $\alpha_i$ is decreasing.
    Then \eqref{PLMS-CDC} implies $C(u,\alpha(u))=a_{i-1}\alpha_i(u)=w_i\alpha_i(u)$, and in a similar manner, $4\int_{a_{i-1}}^{a_i}C(u,\alpha(u))du=2w_i\triangle(a_i)$.\\
    Thus, $\tau_{C_{e,\alpha}}=2\sum_{i=1}^{N}w_i\triangle(a_i)-1$ as desired.
    \item \emph{Spearman's rho}: In order to find $\rho_{C_{e,\alpha}}$, notice that $$\rho_{C_{e,\alpha}}=12\Ex\big[\Ex[UV\mid U]\big]-3=12\Ex\big[U\alpha(U)\big]-3=12\int_0^1u\cdot\alpha(u)\,du-3,$$
    which can again be considered in 2 cases.\\
    \textbf{Case 1:} $\alpha_i$ is increasing.
    Then 
    \begin{align*}
        12\int_{a_{i-1}}^{a_i}u\cdot\alpha(u)du&=12\triangle(a_i)\int_0^1(a_{i-1}+\triangle(a_i)z)z\,dz=12\triangle(a_i)\left[\frac{a_{i-1}}{2}+\frac{\triangle(a_i)}{3}\right] \tag{$z=\alpha_i(u)$}\\
        &=2\triangle(a_i)\left[2a_i+a_{i-1}\right]=2w_i\triangle(a_i)+2(a_i^2-a_{i-1}^2).
    \end{align*}
    \textbf{Case 2:} $\alpha_i$ is decreasing.
    It is similar to the previous case to show that $$12\int_{a_{i-1}}^{a_i}u\cdot\alpha(u)du=2w_i\triangle(a_i)+2(a_i^2-a_{i-1}^2).$$
    Therefore,
    \[\rho_{C_{e,\alpha}}=2\sum_{i=1}^{N}w_i\triangle(a_i)+2\sum_{i=1}^{N}(a_i^2-a_{i-1}^2)-3=2\sum_{i=1}^{N}w_i\triangle(a_i)-1.\qedhere\]
    
    \end{itemize}
\end{proof}
\subsection{Factorizable copulas}
In the above section, we observed that complete dependence copulas $C$ involving PLMS functions satisfy $\tau_C=\rho_C$. Since $\ast$ is a binary operation on $\mathscr{C}$ with many appealing algebraic properties, it is natural to investigate the case of a product between copulas related to PLMS functions, i.e., the product $C\ast D$ where $C=C_{e,\alpha}$ or $C_{\alpha,e}$; and $D=C_{e,\beta}$ or $C_{\beta,e}$ for some $\alpha,\beta\in\mathscr{T}_{\text{PLMS}}$.
\par 
    Even though copulas of the form $C_{\alpha,e}\ast C_{e,\beta}=C_{\alpha,\beta}$ constitute the most general form of such copulas (see \cite{Vit1996}), there is no transparent interpretation of the product in terms of the graphs of $\alpha$ and $\beta$. So we restrict our interest to the cases $C_{e,\alpha}\ast C_{e,\beta}=C_{e,\beta\circ\alpha}$ and $C_{e,\alpha}\ast_{\mathcal{A}}C_{\beta,e}$. Note that $C_{\alpha,e}\ast C_{\beta,e}=C_{\alpha\circ\beta,e}$ is simply the transpose of $C_{e,\beta}\ast C_{e,\alpha}$.
    \par 
    First, let us consider the case $C=C_{e,\alpha}\ast C_{e,\beta}=C_{e,\beta\circ\alpha}$ when $\alpha,\beta\in\mathscr{T}_{\text{PLMS}}$ with the corresponding partitions $[a_i]_{i=0}^{P}$ and $[b_j]_{j=0}^{Q}$, respectively. 
    It is routine to show that $\beta\circ\alpha\in\mathscr{T}_{\text{PLMS}}$ with the corresponding partition $[c_k]_{k=0}^{PQ}$, where 
\begin{equation}\label{composite-1}
    c_k=\begin{cases}
        a_{i-1}+b_j\paren{a_i-a_{i-1}} & \text{if}\ \alpha_i\ \text{is increasing};\\
        a_i-b_{Q-j}\paren{a_i-a_{i-1}} & \text{if}\ \alpha_i\ \text{is decreasing},
    \end{cases} 
\end{equation}
whenever $k=\paren{i-1}Q+j$, $i\in\set{1,2,\dots,P}$, and $j\in\set{0,1,\dots,Q}$ as shown in Figure~\ref{composite-11}.
\begin{figure}
        \centering
        \begin{tikzpicture}[scale=1.75]
            \draw (0,0) rectangle (0.7,1);
            \draw (0.7,0) rectangle (1.5,1);
            \draw (1.5,0) rectangle (2.1,1);
            \draw (2.1,0) rectangle (2.7,1);
            \draw (2.7,0) rectangle (3.2,1);
            \draw[very thick](0,0)--(0.7,1)--(1.5,0);
            \node at (1.8,0.5) {\footnotesize$\dots$};
            \draw[very thick](2.1,0)--(2.7,1);
            \draw[very thick](2.7,0)--(3.2,1);
            \node[below] at (0,0) {\scriptsize$a_{i-1}$};
            \node[below] at (0.7,0) {\scriptsize\begin{tabular}{c}
                $a_{i-1}+$  \\
                $b_1\triangle$ 
            \end{tabular}};
            \node[below] at (1.5,0) {\scriptsize\begin{tabular}{c}
                $a_{i-1}+$  \\
                $b_2\triangle$ 
            \end{tabular}};
            \node[below] at (2.1,0) {\scriptsize\begin{tabular}{c}
                $a_{i-1}+$  \\
                $b_{Q-2}\triangle$ 
            \end{tabular}};
            \node[below] at (2.7,0) {\scriptsize\begin{tabular}{c}
                $a_{i-1}+$  \\
                $b_{Q-1}\triangle$ 
            \end{tabular}};
            \node[below] at (3.2,0) {\scriptsize$a_i$};
            \node[anchor=south east] at (0,0){\footnotesize 0};
            \node[left] at (0,1){\footnotesize 1};
            \node[below] at (1.5,-0.5) {$\alpha_i$ is increasing};
        \end{tikzpicture} \hspace{1em}
        \begin{tikzpicture}[scale=1.75]
            \draw (0,0) rectangle (0.5,1);
            \draw (0.5,0) rectangle (1.1,1);
            \draw (1.1,0) rectangle (1.7,1);
            \draw (1.7,0) rectangle (2.5,1);
            \draw (2.5,0) rectangle (3.2,1);
            \draw[very thick](0,1)--(0.5,0);
            \draw[very thick](0.5,1)--(1.1,0);
            \node at (1.4,0.5) {\footnotesize$\dots$};
            \draw[very thick](1.7,0)--(2.5,1)--(3.2,0);
            \node[below] at (0,0) {\scriptsize$a_{i-1}$};
            \node[below] at (0.5,0) {\scriptsize\begin{tabular}{c}
                $a_i-$  \\
                $b_{Q-1}\triangle$ 
            \end{tabular}};
            \node[below] at (1.1,0) {\scriptsize\begin{tabular}{c}
                $a_i-$  \\
                $b_{Q-2}\triangle$ 
            \end{tabular}};
            \node[below] at (1.7,0) {\scriptsize\begin{tabular}{c}
                $a_i-$  \\
                $b_2\triangle$ 
            \end{tabular}};
            \node[below] at (2.5,0) {\scriptsize\begin{tabular}{c}
                $a_i-$  \\
                $b_1\triangle$ 
            \end{tabular}};
            \node[below] at (3.2,0) {\scriptsize$a_i$};
            \node[anchor=south east] at (0,0){\footnotesize 0};
            \node[left] at (0,1){\footnotesize 1};
            \node[below] at (1.5,-0.5) {$\alpha_i$ is decreasing};
        \end{tikzpicture}
        \caption{Composition of PLMS functions, where $\triangle:=\triangle(a_i)=a_i-a_{i-1}$.}
        \label{composite-11}
    \end{figure}
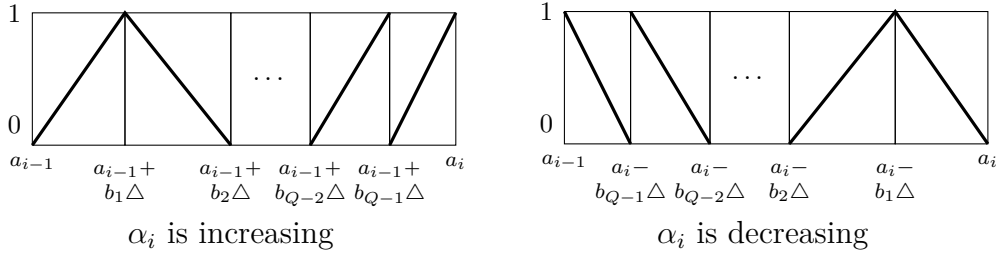
This observation immediately implies that $\tau_{C_{e,\alpha}\ast C_{e,\beta}}=\rho_{C_{e,\alpha}\ast C_{e,\beta}}$ by Theorem~\ref{tau-rho-1}. \par What is more surprising is the multiplicative property of these quantities. This implies that  $\tau$ and $\rho$ can be treated as a monoid homomorphism between the class of $C_{e,\alpha}$, where $\alpha\in\mathscr{T}_{\text{PLMS}}$, under the Markov product and the interval $[-1,1]$ under the usual multiplication.
\begin{thm}\label{tau-rho-2}
    For any $\alpha,\beta\in\mathscr{T}_{\text{PLMS}}$: $$\rho_{C_{e,\alpha}\ast C_{e,\beta}}=\tau_{C_{e,\alpha}\ast C_{e,\beta}}=\tau_{C_{e,\alpha}}\cdot \tau_{C_{e,\beta}}.$$
\end{thm}
Now, we turn to the case $C=C_{e,\alpha}\ast C_{\beta,e}$ when $\alpha,\beta\in\mathscr{T}_{\text{PLMS}}$ with the corresponding partitions $[a_i]_{i=0}^{P}$ and $[b_j]_{j=0}^{Q}$.
Note that $C$ has the mass distributed on the graph $\alpha(x)=\beta(y)$ \cite{product21,CPMS22,Sum2017} and is called a \emph{factorizable copulas}. Additionally, the pointwise values of $C$ are found to be: 
    \begin{equation}\label{PLMS-fac}
        C(u,v) = \begin{cases}
            \min\set{b_{j-1}u+a_iv-a_ib_{j-1},b_ju+a_{i-1}v-a_{i-1}b_j} & \text{ }\ \alpha_i,\beta_j\ \text{same monotonicity};\\
            \max\set{b_{j-1}u+a_{i-1}v-a_{i-1}b_{j-1},b_ju+a_iv-a_ib_j} & \text{ }\ \alpha_i,\beta_j\ \text{diff. monotonicities},
        \end{cases}
    \end{equation}
which is illustrated in Figure~\ref{factor-point}.
\begin{figure}
        \centering
        \begin{tikzpicture}[scale=2]
            \draw (0,0) rectangle (1.5,1);
            \node[below] at (0,0){$a_{i-1}$};
            \node[below] at (1.5,0){$a_{i}$};
            \node[anchor=south east] at (0,0){$b_{j-1}$};
            \node[left] at (0,1){$b_j$};
            \draw[very thick](0,0)--(1.5,1);
            \node[left] at (0.75,0.6){\footnotesize$b_jx+a_{i-1}y-a_{i-1}b_j$};
            \node[right] at (0.75,0.4){\footnotesize$b_{j-1}x+a_{i}y-a_ib_{j-1}$};
            \node[below] at (0.75,-0.25){\scriptsize$\alpha_i,\beta_j$ same monotonicity};
        \end{tikzpicture} \hspace{2ex}
        \begin{tikzpicture}[scale=2]
            \draw (0,0) rectangle (1.5,1);
            \node[below] at (0,0){$a_{i-1}$};
            \node[below] at (1.5,0){$a_{i}$};
            \node[anchor=south east] at (0,0){$b_{j-1}$};
            \node[left] at (0,1){$b_j$};
            \draw[very thick](0,1)--(1.5,0);
            \node[right] at (0.75,0.6){\footnotesize$b_jx+a_iy-a_ib_j$};
            \node[left] at (0.75,0.4){\footnotesize$b_{j-1}x+a_{i-1}y-a_{i-1}b_{j-1}$};
            \node[below] at (0.75,-0.25){\scriptsize$\alpha_i,\beta_j$ different monotonicities};
        \end{tikzpicture}
        \caption{Support and pointwise values of $C_{e,\alpha}\ast C_{\beta,e}$ on $R_{ij}$ where (left) $\alpha_i$ and $\beta_j$ have the same monotonicity, and (right) $\alpha_i$ and $\beta_j$ have different monotonicities.}
        \label{factor-point}
    \end{figure}
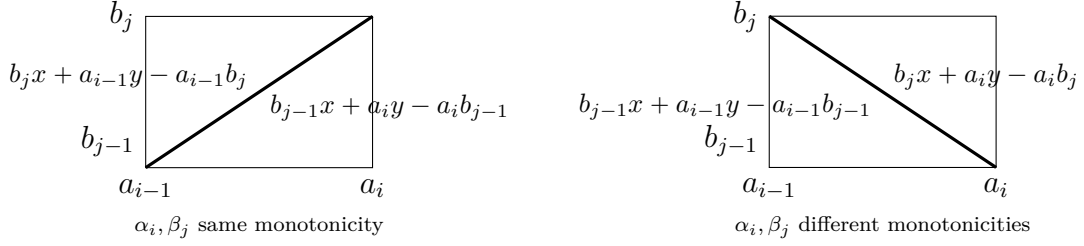
To investigate the concordance measures of factorizable copulas $C_{e,\alpha}\ast C_{\beta,e}$, where $\alpha,\beta\in\mathscr{T}_{\text{PLMS}}$, we start with  formula \eqref{PLMS-fac} and by algebraic manipulations, we obtain an interesting result.
\begin{thm}[$\tau$ and $\rho$ of product between PLMS]\label{tau-rho-3}
        For any $\alpha,\beta\in\mathscr{T}_{\text{PLMS}}$: $$\tau_{C_{e,\alpha}\ast C_{\beta,e}}=\rho_{C_{e,\alpha}\ast C_{\beta,e}}=\tau_{C_{e,\alpha}}\cdot \tau_{C_{\beta,e}}.$$
    \end{thm}
The proofs of Theorems~\ref{tau-rho-2} and \ref{tau-rho-3} are given in the Appendix.
\subsection{Patched copulas}
Suppose that $\mathcal{C}:=\set{C_{ij}}_{i=1,j=1}^{P,Q}$ is a collection of copulas. A new copula can be constructed using $\mathcal{C}$ and a \emph{transformation matrix} or \emph{mass distribution matrix} $T=[t_{ij}]$. This matrix has  $P$ rows and $Q$ columns with the following properties: All entries of $T$ are nonnegative with the total sum 1, constituting a probability measure, and no row or column of $T$ is a zero vector; see \cite{patch16,cfs05}. This matrix provides two partitions of $\I$, namely $[a_k]_{k=0}^{P}$ and $[b_\ell]_{\ell=0}^{Q}$, by $a_0=b_0=0$, $a_k=\sum_{i=1}^{k}\sum_{j=1}^{Q}t_{ij}$, and $b_\ell=\sum_{j=1}^{\ell}\sum_{i=1}^{P}t_{ij}$.
We then assign each $C_{k\ell}$ to a rectangular partition $R_{k\ell}=\paren{a_{k-1},a_k}\times\paren{b_{\ell-1},b_{\ell}}$.
The resulting copula $C$ may be called a \emph{patched copula} of $\mathcal{C}$ with respect to the matrix $T$.
Mathematically, a patched copula of $\mathcal{C}$ with respect to the transformation matrix $T$, denoted by $T(\mathcal{C})$, can be formulated by
\begin{equation}
    T(\mathcal{C})(u,v):=\sum_{i<k,\,j<\ell}t_{ij}+\left(\frac{u-a_{k-1}}{a_k-a_{k-1}}\right)\sum_{j<\ell}t_{kj}+\left(\frac{v-b_{\ell-1}}{b_{\ell}-b_{\ell-1}}\right)\sum_{i<k}t_{i\ell}+t_{k\ell}C_{k\ell}\left(\frac{u-a_{k-1}}{a_k-a_{k-1}},\frac{v-b_{\ell-1}}{b_{\ell}-b_{\ell-1}}\right),
\end{equation}
for all $\paren{u,v}\in R_{k\ell}$.
Observe that rows of $T$ from top to bottom correspond to the horizontal axis of $T(\mathcal{C})$, and columns to the vertical axis.
Equivalently, $T(\mathcal{C})$ can be represented in terms of measure of every rectangle $R=[x_1,x_2]\times[y_1,y_2]\subseteq\I^2$ as
\begin{equation}
    \mu_{T(\mathcal{C})}\paren{R}=\sum_{i=1}^{P}\sum_{j=1}^{Q}t_{ij}\mu_{C_{ij}}\big(E_{ij}\paren{R\cap R_{ij}}\big),
\end{equation}
as shown in Figure \ref{fig:patch}, 
where $E_{ij}\colon\I^2\to\I^2$ \edit{defined by $E_{ij}\paren{u,v}=\paren{F_{\mathcal{U}\brac{a_{i-1},a_i}}(u),F_{\mathcal{U}\brac{b_{j-1},b_j}}(v)}$,
and $F_{\mathcal{U}\brac{a,b}}$ denotes the distribution function of the uniform random variable on $\brac{a,b}$.
}
\begin{figure}
    \centering
    \begin{tikzpicture}[scale=2.5]
        \draw (0,0) rectangle (1,1);
        \begin{scope}[transparency group]
        \begin{scope}[blend mode=multiply]
            \filldraw[fill=gray!40](0.2,0.25) rectangle (0.7,0.75);
            \draw[thick](0,0)--(0.4,0.25);
            \fill[pattern={crosshatch}](0,0.25) rectangle (0.4,0.6);
            \draw[thick](0,1)--(0.4,0.6);
            \draw[thick](0.4,0)--(0.6,0.25)--(0.7,0);
            \draw[thick](0.7,0.25)--(0.4,0.6);
            \draw[thick](0.4,0.25)--(0.7,0.6);
            \draw[thick](0.4,0.6)--(0.7,0.8)--(0.4,0.9)--(0.7,1);
            \fill[pattern={crosshatch}](0.7,0) rectangle (1,0.25);
            \draw[thick](0.7,0.25)--(1,0.4);
            \draw[thick](0.7,0.4)--(1,0.6);
            \draw[thick](0.7,0.6)--(0.85,0.8);
            \draw[thick](0.85,1)--(1,0.8);
        \end{scope}
        \end{scope}
        \foreach \x in {0.4,0.7}{
            \draw[dashed](\x,1)--(\x,0);
        }
        \node[below] at (0,0){\footnotesize$a_0$};
        \node[below] at (0.4,0){\footnotesize$a_1$};
        \node[below] at (0.7,0){\footnotesize$a_2$}; 
        \node[below] at (1,0){\footnotesize$a_3$}; 
        \foreach \y in {0.25,0.6}{
            \draw[dashed](0,\y)--(1,\y);    
        }
        \node[left] at (0,0){\footnotesize$b_0$};
        \node[left] at (0,0.25){\footnotesize$b_1$}; 
        \node[left] at (0,0.6){\footnotesize$b_2$}; 
        \node[left] at (0,1){\footnotesize$b_3$};
        \draw[-latex] (1.25,0.5)--(2,0.5);
        \node at (3,0.5){+};
        \draw[dashed] ($(0,0)+(2.25,0.75)$) rectangle ($0.5*(1,1)+(2.25,0.75)$);
        \begin{scope}[transparency group]
        \begin{scope}[blend mode=multiply]
        \filldraw[fill=gray!40]($0.5*(0.5,0)+(2.25,0.75)$) rectangle ($0.5*(1,1)+(2.25,0.75)$);
        \fill[pattern={crosshatch}]($(0,0)+(2.25,0.75)$) rectangle ($0.5*(1,1)+(2.25,0.75)$);
        \end{scope}
        \end{scope}
        \node[above] at ($(0.25,0.5)+(2.25,0.75)$){\footnotesize weight=$t_{12}$};
        \draw[dashed] ($(0,0)+(3.25,0.75)$) rectangle ($0.5*(1,1)+(3.25,0.75)$);
        \filldraw[fill=gray!40]($0.5*(0.5,0)+(3.25,0.75)$) rectangle ($0.5*(1,3/8)+(3.25,0.75)$);
        \draw[thick]($0.5*(0,1)+(3.25,0.75)$)--($0.5*(1,0)+(3.25,0.75)$);
        \node[above] at ($(0.25,0.5)+(3.25,0.75)$){\footnotesize weight=$t_{13}$};
        \draw[dashed] ($(0,0)+(2.25,-0.25)$) rectangle ($(0.5,0.5)+(2.25,-0.25)$);
        \node[above] at ($(0.25,0.5)+(2.25,-0.25)$){\footnotesize weight=$t_{22}$};
        \filldraw[fill=gray!40]($0.5*(0,0)+(2.25,-0.25)$) rectangle ($0.5*(1,1)+(2.25,-0.25)$);
        \draw[thick]($0.5*(0,0)+(2.25,-0.25)$)--($0.5*(1,1)+(2.25,-0.25)$);
        \draw[thick]($0.5*(0,1)+(2.25,-0.25)$)--($0.5*(1,0)+(2.25,-0.25)$);
        \draw[dashed] ($(0,0)+(3.25,-0.25)$) rectangle ($0.5*(1,1)+(3.25,-0.25)$);
        \node[above] at ($(0.25,0.5)+(3.25,-0.25)$){\footnotesize weight=$t_{23}$};
        \filldraw[fill=gray!40]($0.5*(0,0)+(3.25,-0.25)$) rectangle ($0.5*(1,3/8)+(3.25,-0.25)$);
        \draw[thick]($0.5*(0,0)+(3.25,-0.25)$)--($0.5*(1,0.5)+(3.25,-0.25)$)--($0.5*(0,0.75)+(3.25,-0.25)$)--($0.5*(1,1)+(3.25,-0.25)$);
    \end{tikzpicture}
    \caption{An example of computing a patched copula in terms of measure.}
    \label{fig:patch}
\end{figure}
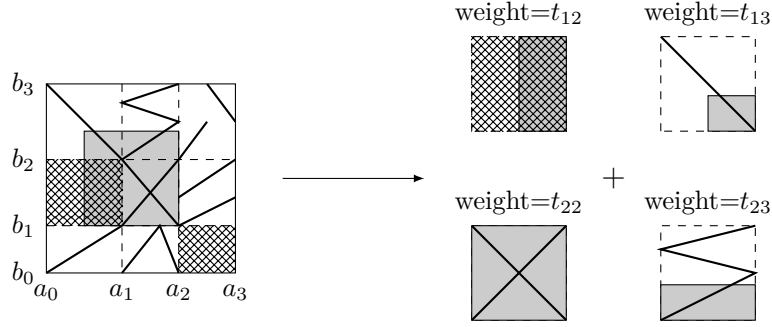
Notice that every copula $C_{e,\alpha}$, where $\alpha\in\mathscr{T}_{\text{PLMS}}$, can be regarded as the patched copula of $M\colon\paren{x,y}\mapsto\min\set{x,y}$ and $W\colon\paren{x,y}\mapsto\max\set{x+y-1,0}$ with respect to the column matrix $T=\begin{bmatrix}
    \triangle(a_1) & \ldots & \triangle(a_P)
\end{bmatrix}^t,$ where $$C_i=\begin{cases}
    M & \text{if}\ \alpha_i\ \text{is increasing};\\
    W & \text{if}\ \alpha_i\ \text{is decreasing}.
\end{cases}$$
Similarly, every factorizable copula $C_{e,\alpha}\ast C_{\beta,e}$, where $\alpha,\beta\in\mathscr{T}_{\text{PLMS}}$, is the patched copula of $\set{C_{ij}}\subseteq\set{M,W}$ with respect to the transformation matrix
\[T=\begin{bmatrix}
    \triangle(a_1)\triangle(b_1) & \triangle(a_1)\triangle(b_2) & \dots & \triangle(a_1)\triangle(b_Q)\\
    \triangle(a_2)\triangle(b_1) & \triangle(a_2)\triangle(b_2) & \dots & \triangle(a_2)\triangle(b_Q)\\
    \vdots & \vdots & \ddots & \vdots \\
    \triangle(a_P)\triangle(b_1) & \triangle(a_P)\triangle(b_2) & \dots & \triangle(a_P)\triangle(b_Q)
\end{bmatrix},\]
where $C_{ij}$ is $M$ or $W$ depending on whether $\paren{\alpha_i,\beta_j}$ have the same or different monotonicity, respectively.
\par 
Now, suppose that $S$ and $T$ are $P\times1$ and $1\times Q$ transformation matrices, respectively, and $\mathcal{C}:=\left\{C_i\right\}_{i=1}^{P}$, $\mathcal{D}:=\left\{D_j\right\}_{j=1}^{Q}$ are collections of copulas.
The formula for $S(\mathcal{C})\ast T(\mathcal{D})$ can be expressed in the following lemma.
\begin{lem}\label{lem6:18Dec}
    Let $[a_i]_{i=0}^{P}$ and $[b_j]_{j=0}^{Q}$ be partitions of $\I$ and $\mathcal{C}=\left\{C_i\right\}_{i=1}^{P}$, $\mathcal{D}=\left\{D_j\right\}_{j=1}^{Q}$ be collections of copulas.
    Setting the transformation matrices $S=[\triangle(a_i)]_{P\times1}$ and $T=[\triangle(b_j)]_{1\times Q}$, as well as $A:=S(\mathcal{C})\ast T(\mathcal{D})$.
    Then for any $(u,v)\in R_{k\ell}$,
    \[A(u,v)=a_{k-1}v+b_{\ell-1}u-a_{k-1}b_{\ell-1}+\triangle(a_k)\triangle(b_{\ell})(C_k\ast D_{\ell})\left(\frac{u-a_{k-1}}{\triangle(a_k)},\frac{v-b_{\ell-1}}{\triangle(b_{\ell})}\right).\]
\end{lem}
\begin{proof}
    Let $(u,v)\in R_{k\ell}$.
    Then for any $t\in\I$, the pointwise values and related partial derivatives of $S(\mathcal{C})$ and $T(\mathcal{D})$ become
    \begin{align*}
        S(\mathcal{C})(u,t)&=a_{k-1}t+\triangle(a_k)C_k\left(\frac{u-a_{k-1}}{\triangle(a_k)},t\right),
        &&T(\mathcal{D})(t,v)=b_{\ell-1}t+\triangle(b_{\ell})D_{\ell}\left(t,\frac{v-b_{\ell-1}}{\triangle(b_{\ell})}\right),\\
        \partial_2S(\mathcal{C})(u,t)&=a_{k-1}+\triangle(a_k)\partial_2C_k\left(\frac{u-a_{k-1}}{\triangle(a_k)},t\right),
        &&\partial_1T(\mathcal{D})(t,v)=b_{\ell-1}+\triangle(b_{\ell})\partial_1D_{\ell}\left(t,\frac{v-b_{\ell-1}}{\triangle(b_{\ell})}\right).
    \end{align*}
    Hence, 
    \begin{align*}
       A(u,v)
        &=a_{k-1}b_{\ell-1}+a_{k-1}\triangle(b_{\ell})\int_{0}^{1}\partial_1D_{\ell}\left(t,\frac{v-b_{\ell-1}}{\triangle(b_{\ell})}\right)dt+b_{\ell-1}\triangle(a_k)\int_{0}^{1}\partial_2C_k\left(\frac{u-a_{k-1}}{\triangle(a_k)},t\right)dt\\
        &\ \ \ +\triangle(a_k)\triangle(b_{\ell})\int_{0}^{1}\partial_2C_k\left(\frac{u-a_{k-1}}{\triangle(a_k)},t\right)\partial_1D_{\ell}\left(t,\frac{v-b_{\ell-1}}{\triangle(b_{\ell})}\right)dt\\
        &=a_{k-1}b_{\ell-1}+a_{k-1}\triangle(b_{\ell})\left(\frac{v-b_{\ell-1}}{\triangle(b_{\ell})}\right)+b_{\ell-1}\triangle(a_k)\left(\frac{u-a_{k-1}}{\triangle(a_k)}\right)\\
        &\ \ \ +\triangle(a_k)\triangle(b_{\ell})(C_k\ast D_{\ell})\left(\frac{u-a_{k-1}}{\triangle(a_k)},\frac{v-b_{\ell-1}}{\triangle(b_{\ell})}\right)\\
        &=a_{k-1}v+b_{\ell-1}u-a_{k-1}b_{\ell-1}+\triangle(a_k)\triangle(b_{\ell})(C_k\ast D_{\ell})\left(\frac{u-a_{k-1}}{\triangle(a_k)},\frac{v-b_{\ell-1}}{\triangle(b_{\ell})}\right). \qedhere
    \end{align*}
\end{proof}
From Lemma~\ref{lem6:18Dec}, it is rather tedious to calculate Kendall's tau and Spearman's rho of $A$, although the result is predictable.
\begin{thm}\label{thm7:18Dec}
    Let $A=S(\mathcal{C})\ast T(\mathcal{D})$ where $S,T,\mathcal{C},$ and $\mathcal{D}$ are as in Lemma~\ref{lem6:18Dec}.
    Then $$\displaystyle\tau_A=\sum_{i=1}^{P}\sum_{j=1}^{Q}(\triangle(a_i))^2(\triangle(b_j))^2\tau_{C_i\ast D_j}\text{ and }\displaystyle\rho_A=\sum_{i=1}^{P}\sum_{j=1}^{Q}(\triangle(a_i))^2(\triangle(b_j))^2\rho_{C_i\ast D_j}.$$
\end{thm}
The proof of Theorem \ref{thm7:18Dec} is given in the Appendix. Furthermore, this statement is a generalization of Theorems~\ref{tau-rho-1} and \ref{tau-rho-3}, which provides an alternative approach for the fact that $\tau_C=\rho_C$ for $C=C_{e,\alpha},C_{\beta,e}$, and $C_{e,\alpha}\ast C_{\beta,e}$ where $\alpha,\beta\in\mathscr{T}_{\text{PLMS}}$.
\par Finally, we observe from Theorem \ref{thm7:18Dec} that for any sequences of transformation matrices $\paren{S_n}$ and $\paren{T_n}$ in which either $P_n=\#\text{rows of}\ S_n$ or $Q_n=\#\text{columns of}\ T_n$ tends to $\infty$ and the biggest entry of the corresponding matrices converges to zero, the associated sequences of Kendall's $\tau$ and Spearman's $\rho$, namely, $\tau_{A_n}$ and $\rho_{A_n}$ with $A_n=S_n(\mathcal{C}^n)\ast T_n(\mathcal{D}^n)$, converge to zero for every sequence of collections of copulas $\mathcal{C}^n$ and $\mathcal{D}^n$ corresponding to $S_n$ and $T_n$, respectively. 
A similar observation is formalized in Corollary \ref{approx-pi-1}.
\subsection{Implicit dependence copulas}
Both complete dependence and factorizable copulas are special cases of a large subclass of singular copulas, namely \emph{implicit dependence copulas}, which was studied extensively in the literature, e.g., \cite{CPMS22,Sum2017,Join2025}.
\edit{This includes the concordance measures of some hairpin copulas, a specific type of diagonal and implicit dependence copulas, see e.g.,~\cite{FSW26}.}
Expressing a copula of this type in terms of a generalized Markov product between left and right invertible copulas is given in \cite{product21,CPMS22}.
Note that $\ast_{\mathcal{A}}$ is a binary operation with the identity $M$, but it may not have the associative property in general.
Thus, it is rather complicated to determine the equality of $\tau_C$ and $\rho_C$, even in the case $C=C_{e,\alpha}\ast_{\mathcal{A}}C_{\beta,e}$, $\alpha,\beta\in\mathscr{T}_{\text{PLMS}}$, and $A_t=A$ for all $t$.
However, in our setting, the collection of joining copulas $\mathcal{A}=\set{A_t}_{t\in\I}$ is supposed to be defined as a piecewise function with respect to $t$, i.e., $A_t=C_{\ell}$ for every $t\in\paren{t_{\ell-1},t_{\ell}}$ whenever $[0,1]=\bigcup_{\ell=1}^{L}\paren{t_{\ell-1},t_\ell}$ almost surely.
Hence, the pointwise value of $C$, for $(u,v)\in\paren{a_{r-1},a_r}\times\paren{b_{s-1},b_s}$ such that $\paren{\alpha(u),\beta(v)}\in\paren{t_{p-1},t_p}\times\paren{t_{q-1},t_q}$, is determined by
\begin{equation}\label{eq3:9Apr}
        C(u,v)=\begin{cases}
        \begin{aligned}
            &\alpha(u)\brac{C_p\paren{m^+(\alpha_r),m^+(\beta_s)}-C_p\paren{m^-(\alpha_r),m^+(\beta_s)}}\\
            &+\beta(v)\brac{C_q\paren{m^-(\alpha_r),m^+(\beta_s)}-C_q\paren{m^-(\alpha_r),m^-(\beta_s)}}\\
            &-\sum_{k=1}^{L-1}\brac{C_{k+1}\paren{m_{k,p}(\alpha_r),m_{k,q}(\beta_s)}-C_{k}\paren{m_{k,p}(\alpha_r),m_{k,q}(\beta_s)}}t_k\\
            &+C_L\paren{m^-(\alpha_r),m^-(\beta_s)}\\
            &\phantom{...}
        \end{aligned}
        & \text{if}\ \alpha(u)\le\beta(v);\\
        \begin{aligned}
            &\beta(v)\brac{C_q\paren{m^+(\alpha_r),m^+(\beta_s)}-C_q\paren{m^+(\alpha_r),m^-(\beta_s)}}\\
            &+\alpha(u)\brac{C_p\paren{m^+(\alpha_r),m^-(\beta_s)}-C_p\paren{m^-(\alpha_r),m^-(\beta_s)}}\\
            &-\sum_{k=1}^{L-1}\brac{C_{k+1}\paren{m_{k,p}(\alpha_r),m_{k,q}(\beta_s)}-C_{k}\paren{m_{k,p}(\alpha_r),m_{k,q}(\beta_s)}}t_k\\
            &+C_L\paren{m^-(\alpha_r),m^-(\beta_s)}
        \end{aligned}
        & \text{if}\ \alpha(u)\ge\beta(v),
    \end{cases}
    \end{equation}
    where $m^+(f):=\amax(f)$, $m^-(f):=\amin(f)$, and $m_{k,h}(f)=\begin{cases}
        m^+(f) & \text{if}\ k<h; \\
        m^-(f) & \text{if}\ k\ge h.
    \end{cases}$\\
In particular, if $A_t=A$ for all $t\in\I$, then the summand in \eqref{eq3:9Apr} is zero and hence \eqref{eq3:9Apr} becomes
\begin{equation}\label{eq4:9Apr}
        C(u,v)=\begin{cases}
        \begin{aligned}
        &\alpha(u)\brac{A\paren{m^+(\alpha_r),m^+(\beta_s)}-A\paren{m^-(\alpha_r),m^+(\beta_s)}}\\
        &+\beta(v)\brac{A\paren{m^-(\alpha_r),m^+(\beta_s)}-A\paren{m^-(\alpha_r),m^-(\beta_s)}}\\
        &+A\paren{m^-(\alpha_r),m^-(\beta_s)}    
        \end{aligned}
    & \text{if}\ \alpha(u)\le\beta(v);\\
        \begin{aligned}
        &\beta(v)\brac{A\paren{m^+(\alpha_r),m^+(\beta_s)}-A\paren{m^+(\alpha_r),m^-(\beta_s)}}\\
        &+\alpha(u)\brac{A\paren{m^+(\alpha_r),m^-(\beta_s)}-A\paren{m^-(\alpha_r),m^-(\beta_s)}}\\
        &+A\paren{m^-(\alpha_r),m^-(\beta_s)}    
        \end{aligned}
    & \text{if}\ \alpha(u)\ge\beta(v).
    \end{cases}
    \end{equation}
    \begin{ex}\label{sec3-ex2}
        A simple case of \eqref{eq3:9Apr} is $\paren{C_{e,\alpha},C_{\beta,e}}=\paren{C_{\theta_1}^{\oplus,\ominus},\paren{C_{\theta_2}^{\oplus,\ominus}}^t}$, as shown exemplarily in Figure~\ref{fig2:23Oct}.
    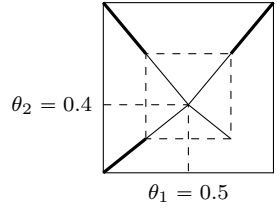
\begin{figure}
        \centering
        \begin{tikzpicture}[scale=2.25]
            \draw (0,0) rectangle (1,1);
            \draw[dashed](0,0.4)--(0.5,0.4)--(0.5,0) node[below]{\scriptsize$\theta_1=0.5$};
            \node[left] at (0,0.4){\scriptsize$\theta_2=0.4$};
            \draw[dashed,thin](0.25,0.2) rectangle (0.75,0.7);
            \draw (0.25,0.2)--(0.5,0.4)--(0.75,0.2);
            \draw (0.25,0.7)--(0.5,0.4)--(0.75,0.7);
            \draw[very thick] (0,0)--(0.25,0.2);
            \draw[very thick] (0,1)--(0.25,0.7);
            \draw[very thick] (0.75,0.7)--(1,1);
        \end{tikzpicture}
        \caption{The product $C_{\theta_1}^{\oplus,\ominus}\ast_{\mathcal{A}}(C_{\theta_2}^{\oplus,\ominus})^t$ and $A_t=\begin{cases}
            M & \text{if}\ t\in\left[0,\frac{1}{2}\right],\\
            \Pi & \text{if}\ t\in(\frac{1}{2},1].
        \end{cases}$}
        \label{fig2:23Oct}
    \end{figure}
    Note that $\tau_{C_{\theta_i}^{\oplus,\ominus}}=\rho_{C_{\theta_i}^{\oplus,\ominus}}=2\theta_i-1$ for $i=1,2$, see Example \ref{sec3-ex1}. 
    In addition, by setting $D=C_{\theta_1}^{\oplus,\ominus}\ast_{\mathcal{A}}\paren{C_{\theta_2}^{\oplus,\ominus}}^t$, the concordance measures of $D$ in the case $A_t=A$ for all $t$ can be computed straightforwardly as
    \[\tau_D=\rho_D=1-2\paren{\theta_1+\theta_2}+4A\paren{\theta_1,\theta_2}.\]
    On the other hand, if there is $s\in\paren{0,1}$ such that $A_t=\begin{cases}
        A_1 & \text{if}\ t\in[0,s];\\
        A_2 & \text{if}\ t\in(s,1],
    \end{cases}$
    the concordance measures of $D$ are more complicated, and the inequality $\tau_D\ne\rho_D$ occurs unless $A_1\paren{\theta_1,\theta_2}=A_2\paren{\theta_1,\theta_2}$.
    More precisely,
    \begin{align*}
        \tau_D&=1-2\paren{\theta_1+\theta_2}+4A_1\paren{\theta_1,\theta_2}+4\brac{A_2\paren{\theta_1,\theta_2}-A_1\paren{\theta_1,\theta_2}}(1-s)^2;\\
        \rho_D&=1-2\paren{\theta_1+\theta_2}+4A_1\paren{\theta_1,\theta_2}+4\brac{A_2\paren{\theta_1,\theta_2}-A_1\paren{\theta_1,\theta_2}}(1-s)^3.
    \end{align*}
    A numerical example in this case is shown in Figure~\ref{fig2:23Oct}, for which $\tau_D=\frac{3}{5}\ne\frac{7}{10}=\rho_D$.
    \end{ex}
\section{Intriguing properties of PLMS and related copulas}\label{sec:PLMS-prop}
In this section, we investigate other dependence measures and more properties of copulas created from PLMS.
\subsection{Chatterjee's correlation coefficient}
Besides relying on concordance, a measure of dependence can also be created with the intention to detect a relation between random variables, regardless of the direction, as done by Chatterjee \cite{Chat21}. Recall that the \emph{Chatterjee's rank correlation} in terms of copula $C$, denoted by $\xi_C$, is formulated by $$\xi_C=6\iint_{\I^2}\paren{\partial_1C(u,v)}^2dudv-2.$$
Even when $C$ is not absolutely continuous, the partial derivative $\partial_1 C(u,v)$ exists for Lebesgue-almost every $u$ and
$\partial_1 C(u,v)$ is a version of the regular conditional distribution function
$v \mapsto \Pr(V \le v \mid U = u)$.
Thus, for almost every $u \in [0,1]$,
$v \mapsto \partial_1 C(u,v)$ is the conditional distribution of $V$ given $U=u$. Now, if $\alpha\in\mathscr{T}_{\text{PLMS}}$, it is not hard to see that $\xi_{C_{e,\alpha}}=1$ by the properties of conditional expectation and $V=\alpha(U)$. Conversely, the coefficient $\xi_{C_{\alpha,e}}$ can be formulated in the following statement.
\begin{thm}\label{chat-1}
    Let $\alpha\in\mathscr{T}_{\text{PLMS}}$ with a partition $[a_i]_{i=0}^{N}$.
    Then \[\xi_{C_{\alpha,e}}=1-3\sum_{i=1}^{N}a_ia_{i-1}(a_i-a_{i-1}).\]
\end{thm}
\begin{proof}
    Since $\partial_1C_{\alpha,e}(u,v)=\partial_2C_{e,\alpha}(v,u)$ and the value is a constant on each subregion of $(a_{i-1},a_i)\times\I$ (see Figure \ref{alpha_i-diff}). 
    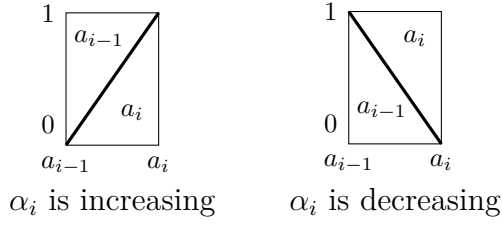
\begin{figure}
        \centering
        \begin{tikzpicture}[scale=1.75]
            \draw (0,0) rectangle (0.7,1);
            \draw[very thick](0,0)--(0.7,1);
            \node[below] at (0,0) {\footnotesize$a_{i-1}$};
            \node[below] at (0.7,0) {\footnotesize$a_{i}$};
            \node[anchor=south east] at (0,0){\footnotesize 0};
            \node[left] at (0,1){\footnotesize 1};
            \node at (0.25,0.8) {\footnotesize$a_{i-1}$};
            \node at (0.5,0.25) {\footnotesize$a_{i}$};
            \node[below] at (0.35,-0.25) {$\alpha_i$ is increasing};
        \end{tikzpicture} \hspace{1em}
        \begin{tikzpicture}[scale=1.75]
            \draw (0,0) rectangle (0.7,1);
            \draw[very thick](0,1)--(0.7,0);
            \node[below] at (0,0) {\footnotesize$a_{i-1}$};
            \node[below] at (0.7,0) {\footnotesize$a_{i}$};
            \node[anchor=south east] at (0,0){\footnotesize 0};
            \node[left] at (0,1){\footnotesize 1};
            \node at (0.25,0.25) {\footnotesize$a_{i-1}$};
            \node at (0.5,0.8) {\footnotesize$a_i$};
            \node[below] at (0.35,-0.25) {$\alpha_i$ is decreasing};
        \end{tikzpicture}
        \caption{$\partial_2C_{e,\alpha}$ for $\alpha\in\mathscr{T}_{\text{PLMS}}$.}
        \label{alpha_i-diff}
    \end{figure}
    This implies that
    \begin{align*}
    \xi_{C_{\alpha,e}}&=6\sum_{i=1}^{N}\iint_{(a_{i-1},a_i)\times\I}\paren{\partial_2C_{e,\alpha}(u,v)}^2dudv-2=6\sum_{i=1}^{N}\paren{a_{i-1}^2\cdot\frac{\triangle(a_i)}{2}+a_i^2\cdot\frac{\triangle(a_i)}{2}}-2\\
    &=3\sum_{i=1}^{N}(a_i^3-a_{i-1}^3)-3\sum_{i=1}^{N}(a_{i-1}a_i^2-a_ia_{i-1}^2)-2=1-3\sum_{i=1}^{N}a_ia_{i-1}\triangle(a_i).\qedhere
    \end{align*}
    \end{proof}
As mentioned earlier, $\xi_{C_{e,\beta\circ\alpha}}=1=\xi_{C_{e,\beta}}\cdot \xi_{C_{e,\alpha}}$.
The following statement shows that Chatterjee's rank correlation for the copula $C_{\beta\circ\alpha,e}$ has the same behavior as $C_{e,\beta\circ\alpha}$. 
The proof of this statement is given in the Appendix.
\begin{thm}\label{chat-2}
    Let $\alpha,\beta\in\mathscr{T}_{\text{PLMS}}$.
    Then: $$\xi_{C_{\beta,e}\ast C_{\alpha,e}}=\xi_{C_{\beta\circ\alpha,e}}=\xi_{C_{\beta,e}}\cdot \xi_{C_{\alpha,e}}.$$
\end{thm}
Recall from Theorem \ref{chat-1} that when $U,V\sim\mathcal{U}(0,1)$ and $V=\alpha(U)$ a.s.\ with $\alpha\in\mathscr{T}_{\text{PLMS}}$, $\xi\paren{V,U}$ can be expressed explicitly in terms of the partition generated by $\alpha$.
In general, if $U,V\sim\mathcal{U}(0,1)$ and $C_{U,V}=C_{e,\alpha}\ast C_{\beta,e}$ for some $\alpha,\beta\in\mathscr{T}_{\text{PLMS}}$, it can be shown that Chatterjee's rank correlations of $\paren{U,V}$ and $(V,U)$ have a very similar form as shown below.
Hence, the proof of Theorem \ref{chat-3} can be omitted.
\begin{thm}\label{chat-3}
    Let $\alpha,\beta\in\mathscr{T}_{\text{PLMS}}$ with the corresponding partitions $[a_i]_{i=0}^{P}$ and $[b_j]_{j=0}^{Q}$, and $U,V\sim\mathcal{U}(0,1)$ satisfying $C_{U,V}=C_{e,\alpha}\ast C_{\beta,e}$.
    Then:
    \begin{align*}
        \xi(U,V)&=1-3\sum_{j=1}^{Q}(b_j-b_{j-1})b_jb_{j-1}\quad\text{and}\quad 
        \xi(V,U)=1-3\sum_{i=1}^{P}(a_i-a_{i-1})a_ia_{i-1}.
    \end{align*}
\end{thm}

\subsection{Convergence of sequences of copulas from PLMS}\label{sec_convergence} 
On the space of copulas $\mathscr{C}$, there exist various metrics, see \cite{MO&cop09,approx10,metric11}. We use:
\begin{align*}
    d_{\infty}(A,B)&=\sup_{(u,v)\in\I^2}\left|A(u,v)-B(u,v)\right|,\\
    D_k^k(A,B)&=\iint_{\I^2}\left|K_A(u,[0,v])-K_B(u,[0,v])\right|^kdudv=\iint_{\I^2}\left|\partial_1A(u,v)-\partial_1B(u,v)\right|^kdudv,\\
    D_{\partial}(A,B)&=D_1(A,B)+D_1(A^T,B^T).
\end{align*}
Here, $K_C$ is called the \emph{regular conditional distribution} or \emph{Markov kernel} of a copula $C$.
Moreover, it is shown in \cite{metric11} that the convergence behaviors between $D_1$ and $D_2$, two of the most commonly used metrics on $\mathscr{C}$, are the same. Hence, in the sense of topology, $d_{\infty}$, which is the coarsest metric in our list, induces the \emph{weak topology} on $\mathscr{C}$.
On the other hand, $D_1$, as well as $D_2$, induces the \emph{strong topology} on $\mathscr{C}$. 
Note that $D_1$ and $D_2$ are asymmetric, where the symmetric one is $D_{\partial}$, so they can detect a complete dependence relation of $C=C_{(U,V)}$ only in the case where $V$ is a function of $U$. 
Combining Theorems~\ref{tau-rho-1} and  \ref{chat-1} with the notion of the  Riemann sum yields interesting consequences for a sequence of copulas generated by PLMS.
\begin{cor}\label{approx-pi-1}
        Let $(\alpha_n)$ be a sequence in $\mathscr{T}_{\text{PLMS}}$ such that 
        \[\text{mesh}\,(\alpha_n):=\max\left\{a_i^n-a_{i-1}^n:[a_i^n]_{i=0}^{N_n}\ \text{is a partition of}\ \I\ \text{for}\ \alpha_n\right\}\to0\quad\text{as}\ n\to\infty.\]
        Then: $$\disp\lim\limits_{n\to\infty}\tau_{C_{e,\alpha_n}}=\lim\limits_{n\to\infty}\xi_{C_{\alpha_n,e}}=0.$$
    \end{cor}
\begin{proof}
    From the expressions of $\tau_{C_{e,\alpha_n}}$ and $\xi_{C_{\alpha_n,e}}$ in Theorems~\ref{tau-rho-1} and \ref{chat-1}, we have the following inequalities:
    \begin{align*}
        \sum_{i=1}^{N_n}2a_{i-1}^n\triangle(a_i^n)\le\: &\tau_{C_{e,\alpha_n}}+1\le\sum_{i=1}^{N_n}2a_i^n\triangle(a_i^n),\\
        \sum_{i=1}^{N_n}3\paren{a_{i-1}^n}^2\triangle(a_i^n)\le\: &1-\xi_{C_{\alpha_n,e}}\le\sum_{i=1}^{N_n}3\paren{a_i^n}^2\triangle(a_i^n).
    \end{align*}
    Now, as $n\to\infty$, the assumption of a vanishing mesh implies that the bounds of $\tau_{C_{e,\alpha_n}}+1$ approach the Riemann integral $\int_0^12x\,dx=1$, while the bounds of $1-\xi_{C_{\alpha_n,e}}$ tend to $\int_0^13x^2\,dx=1$.
    Hence, the Squeezing Lemma gives $\disp\lim\limits_{n\to\infty}\tau_{C_{e,\alpha_n}}=0=\lim\limits_{n\to\infty}\xi_{C_{\alpha_n,e}}$. 
\end{proof}
    Observe that for every $C\in\mathscr{C}$, $\xi_C=6\iint_{\I^2}\paren{\partial_1C(u,v)-v}^2dudv=6D_2^2\paren{C,\Pi}$.
    Thus, the sequence $\paren{C_{\alpha_n,e}}$ where $\alpha_n\in\mathscr{T}_{\text{PLMS}}$ and $\text{mesh}\paren{\alpha_n}\to0$ as $n\to\infty$ approaches $\Pi$ with respect to the strong topology by Corollary \ref{approx-pi-1}.
    However, $\paren{C_{e,\alpha_n}}\nrightarrow\Pi$ under the metric $D_1$ \cite{metric11}, so $\Pi$ cannot be approximated by sequences of $\paren{C_{e,\alpha_n}}$ or $\paren{C_{\alpha_n,e}}$ under the finer topology than the strong one.
    On the other hand, if $\paren{\alpha_n}$ and $\paren{\beta_n}$ are sequences in $\mathscr{T}_{\text{PLMS}}$ such that $\text{mesh}(\alpha_n)\to0$ and $\text{mesh}(\beta_n)\to0$ as $n\to\infty$, apply a similar strategy as the proof of Corollary \ref{approx-pi-1} to Theorem \ref{chat-3}.
    \begin{cor}[\cite{patch16}]\label{approx-pi-2}
        Let $\paren{\alpha_n}$ and $\paren{\beta_n}$ be sequences in $\mathscr{T}_{\text{PLMS}}$ such that $\text{mesh}(\alpha_n)\to0$ and $\text{mesh}(\beta_n)\to0$ as $n\to\infty$.
        Then the sequence $\paren{C_n}:=\paren{C_{e,\alpha_n}\ast C_{\beta_n,e}}$ is an approximation of $\Pi$ under the topology induced by the $\partial$-metric (and so it holds for the weak and strong topologies).
    \end{cor}
    Corollary \ref{approx-pi-2} leads to some nice consequences. In other words, the convergence of every dependence measure defined globally in our studies holds, which follows easily by a straightforward implication between modes of convergence on the class of copulas \cite{MO&cop09,approx10}.
    Moreover, factorizable copulas can be regarded as the other side of the shuffle of min \cite{2-shuffle-Luroth22}. So Corollaries \ref{approx-pi-1} and \ref{approx-pi-2} are roughly analogous results on the convergence of the double shuffle of the comonotonic copula \cite{2-shuffle-Luroth22}.
\subsection{Tail dependence}
To investigate and communicate the extremal behavior of a pair of random variables, the notion of \emph{tail dependence} is a classical concept, see \cite{QRM15,cop06}. When $X$ and $Y$ are random variables with copula $C$, the \emph{lower} and \emph{upper tail dependence coefficients} of $C$ are defined by
\[\lambda^L_C:=\lim\limits_{t\searrow0}\dfrac{C(t,t)}{t}\quad\text{and}\quad \lambda^U_C:=\lim\limits_{t\nearrow1}\dfrac{1-2t+C\paren{t,t}}{1-t}.\] 
Notice that if $\alpha\in\mathscr{T}_{\text{PLMS}}$ has the partition $[a_i]_{i=0}^{N}$, only the first $\alpha_1$ and last $\alpha_N$ are instrumental for the lower and upper tail dependence of $C_{e,\alpha}$, respectively, where the pointwise values of $C_{e,\alpha}$ on $(0,a_1)\times\I$ and $(a_{N-1},1)\times\I$ are shown in Figure \ref{tail-11}. 
\begin{figure}
        \centering
        \begin{tikzpicture}[scale=1.75]
            \draw (0,0) rectangle (0.7,1);
            \draw[very thick](0,0)--(0.7,1);
            \node[below] at (0,0) {\footnotesize$0$};
            \node[below] at (0.7,0) {\footnotesize$a_1$};
            \node[anchor=south east] at (0,0){\footnotesize 0};
            \node[left] at (0,1){\footnotesize 1};
            \node[left] at (0.3,0.65){\footnotesize $x$};
            \node[right] at (0.3,0.25){\footnotesize $a_1y$};
            \node[below] at (0.35,-0.25) {$\alpha_1$ is increasing};
        \end{tikzpicture} \hspace{1em}
        \begin{tikzpicture}[scale=1.75]
            \draw (0,0) rectangle (0.7,1);
            \draw[very thick](0,1)--(0.7,0);
            \node[below] at (0,0) {\footnotesize$0$};
            \node[below] at (0.7,0) {\footnotesize$a_1$};
            \node[anchor=south east] at (0,0){\footnotesize 0};
            \node[left] at (0,1){\footnotesize 1};
            \node[left] at (0.3,0.25){\footnotesize $0$};
            \node[right] at (0.2,0.65){\footnotesize\begin{tabular}{c}
                $x+a_{1}y$\\
                $-a_{1}$
            \end{tabular}};
            \node[below] at (0.35,-0.25) {$\alpha_1$ is decreasing};
        \end{tikzpicture} \hspace{1em}
        \begin{tikzpicture}[scale=1.75]
            \draw (0,0) rectangle (0.7,1);
            \draw[very thick](0,0)--(0.7,1);
            \node[below] at (0,0) {\footnotesize$a_{N-1}$};
            \node[below] at (0.7,0) {\footnotesize$1$};
            \node[anchor=south east] at (0,0){\footnotesize 0};
            \node[left] at (0,1){\footnotesize 1};
            \node[left] at (0.45,0.65){\footnotesize\begin{tabular}{c}
                $x+a_{N-1}y$\\
                $-a_{N-1}$
            \end{tabular}};
            \node[right] at (0.35,0.25){\footnotesize $y$};
            \node[below] at (0.35,-0.25) {$\alpha_N$ is increasing};
        \end{tikzpicture} \hspace{1em}
        \begin{tikzpicture}[scale=1.75]
            \draw (0,0) rectangle (0.7,1);
            \draw[very thick](0,1)--(0.7,0);
            \node[below] at (0,0) {\footnotesize$a_{N-1}$};
            \node[below] at (0.7,0) {\footnotesize$1$};
            \node[anchor=south east] at (0,0){\footnotesize 0};
            \node[left] at (0,1){\footnotesize 1};
            \node[left] at (0.35,0.35){\footnotesize $a_{N-1}y$};
            \node[right] at (0.35,0.65){\footnotesize $x+y-1$};
            \node[below] at (0.35,-0.25) {$\alpha_N$ is decreasing};
        \end{tikzpicture}
        \caption{Pointwise value of $C_{e,\alpha}$ on $(0,a_1)\times\I$ (two sub-figures on the left) and $(a_{N-1},1)\times\I$ (two sub-figures on the right).}
        \label{tail-11}
    \end{figure}
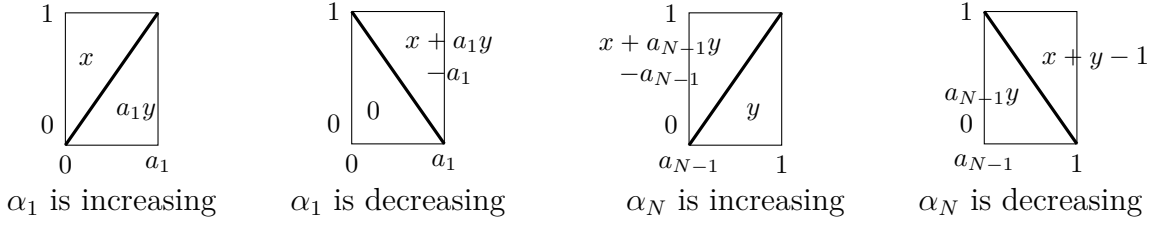
Hence, the tail dependence of $C_{e,\alpha}$ can be expressed in terms of the corresponding endpoints from the partition of $\alpha$.
\begin{thm}\label{tail-1}
    Let $\alpha\in\mathscr{T}_{\text{PLMS}}$ with corresponding partition $[a_i]_{i=0}^{N}$.
    Then the lower and upper  tail dependence coefficients of $C_{e,\alpha}$ can be expressed as
    \[\lambda^L_{C_{e,\alpha}}=\underset{0\le x\le a_1}{\amax}\,\alpha_1(x)\quad\text{and}\quad
    \lambda^U_{C_{e,\alpha}}=1-\underset{a_{N-1}\le x\le 1}{\amin}\,\alpha_N(x).\]
\end{thm}
\begin{proof}
    From \eqref{PLMS-CDC}, the pointwise value of $C_{e,\alpha}$ at $(x,x)\in (0,a_1)\times\I$ is
    \[C_{e,\alpha}(x,x)=\begin{cases}
        \min\set{x,a_1x}=a_1x & \text{if}\ \alpha_1\ \text{is increasing};\\
        \max\set{0,(1+a_1)x-a_1} & \text{if}\ \alpha_1\ \text{is decreasing},
    \end{cases}\]
    where $C_{e,\alpha}(x,x)=0$ for $x<\frac{a_1}{1+a_1}<a_1$.
    On the other hand, for $(x,x)\in (a_{N-1},1)\times\I$,
    \[C_{e,\alpha}(x,x)=\begin{cases}
        \min\set{(1+a_{N-1})x-a_{N-1},x}=x-a_{N-1}(1-x) & \text{if}\ \alpha_N\ \text{is increasing};\\
        \max\set{a_{N-1}x,2x-1} & \text{if}\ \alpha_N\ \text{is decreasing},
    \end{cases}\]
    where $C_{e,\alpha}(x,x)=2x-1$ for $x>\frac{1}{2-a_{N-1}}>a_{N-1}$. 
    Therefore, 
    \begin{align*}
        \lambda^L_{C_{e,\alpha}}&=\lim\limits_{u\searrow0}\frac{C_{e,\alpha}(u,u)}{u}=\begin{cases}
            \lim\limits_{u\searrow0}\frac{a_1u}{u}=a_1 & \text{if}\ \alpha_1\ \text{is increasing};\\
            \lim\limits_{u\searrow0}\frac{0}{u}=0 & \text{if}\ \alpha_1\ \text{is decreasing},
        \end{cases}\\
        \lambda^U_{C_{e,\alpha}}&=\lim\limits_{u\nearrow1}\frac{1-2u+C_{e,\alpha}(u,u)}{1-u}=\begin{cases}
            \lim\limits_{u\nearrow1}\frac{1-u-a_{N-1}(1-u)}{1-u}=1-a_{N-1} & \text{if}\ \alpha_N\ \text{is increasing};\\
            \lim\limits_{u\nearrow1}\frac{1-2u+(2u-1)}{1-u}=0 & \text{if}\ \alpha_N\ \text{is decreasing}.
        \end{cases} \qedhere
    \end{align*}
\end{proof}    
Using the pointwise values of $C=C_{e,\alpha}\ast C_{\beta,e}$ in \eqref{PLMS-fac}, the tail dependence coefficients of $C$ can be derived similarly to Theorem~\ref{tail-1}, however, with slightly more algebraic work.
\begin{thm}\label{tail-2}
        Let $\alpha,\beta\in\mathscr{T}_{\text{PLMS}}$ with corresponding partitions $[a_i]_{i=0}^{P}$ and $[b_j]_{j=0}^{Q}$. Then:
        \begin{align}
            \lambda^L_{C_{e,\alpha}\ast C_{\beta,e}}&=\begin{cases}
            \min\set{a_1,b_1} & \text{if}\ \alpha_1,\beta_1\ \text{have the same monotonicity};\\
            0 & \text{if}\ \alpha_1,\beta_1\ \text{have different monotonicities,}
        \end{cases} \quad\text{and}\\
        \lambda^U_{C_{e,\alpha}\ast C_{\beta,e}}&=\begin{cases}
            1-\max\set{a_{P-1},b_{Q-1}} & \text{if}\ \alpha_P,\beta_Q\ \text{have the same monotonicity};\\
            0 & \text{if}\ \alpha_P,\beta_Q\ \text{have different monotonicities}.
        \end{cases}
        \end{align}
    \end{thm}

\section{Conclusion and further research}
In Section \ref{sec:PLMS}, we have shown that a subclass of factorizable copulas of the form $C=C_{e,\alpha}\ast C_{\beta,e}$, where $\alpha,\beta\in\mathscr{T}_{\mathrm{PLMS}}$, satisfies $\tau_C=\rho_C$. We have further investigated other measures of dependence, including Chatterjee's rank correlation coefficient and the tail dependence coefficients. Viewing the dependence measures $(\tau,\rho,\xi,\lambda^L,\lambda^U)$ as functionals on the space of copulas, we have established the following results:
\begin{itemize}
    \item On the class of one-sided invertible copulas generated by PLMS functions, $\tau$, $\rho$, and $\xi$ are multiplicative; that is, they are homomorphisms into the multiplicative semigroups $([-1,1],\cdot)$ and $([0,1],\cdot)$, respectively.
    \item $\tau$ (and thus $\rho$) for factorizable copulas is separable, while $\xi$ is absorbed to the right operand of the Markov product;
    \item $\lambda_{C_{e,\alpha}\ast C_{\beta,e}}^L$ cannot be attained from $\lambda_{C_{e,\alpha}}^L$ and $\lambda_{C_{\beta,e}}^L$ alone. The same goes for $\lambda_{C_{e,\alpha}\ast C_{\beta,e}}^U$.
\end{itemize}
Furthermore, the second finding implies that the independence copula can be approximated, under the $D_\partial$-metric, by any sequence of factorizable copulas whose meshes of both corresponding PLMS functions approach $0$. 
\par
When attempting to further enlarge the class of copulas with tau-rho--equality, we found that generalized products of PLMS copulas of the form $C_{e,\alpha}\ast_{\mathcal{A}}C_{\beta,e}$, where $\alpha,\beta\in\mathscr{T}_{\mathrm{PLMS}}$ and $\mathcal{A}$ is a parametric family of copulas, do not, in general, preserve this property. This observation leads to the following open problems for a given pair of PLMS functions $\alpha$ and $\beta$:
\emph{Which collections of joining copulas $\mathcal{A}=\set{A_t}_{t\in[0,1]}$ satisfy $\tau_{C_{e,\alpha}\ast_{\mathcal{A}}C_{\beta,e}}=\rho_{C_{e,\alpha}\ast_{\mathcal{A}}C_{\beta,e}}$?} Second: \emph{How large can the discrepancy $\left|\tau_{C_{e,\alpha}\ast_{\mathcal{A}}C_{\beta,e}}
-\rho_{C_{e,\alpha}\ast_{\mathcal{A}}C_{\beta,e}}\right|$ be as $\mathcal{A}$ varies?} Equivalently, for a given value of
$\tau_{C_{e,\alpha}\ast_{\mathcal{A}}C_{\beta,e}}$, what is the range of possible values of
$\rho_{C_{e,\alpha}\ast_{\mathcal{A}}C_{\beta,e}}$?

\section{Appendix}
\subsection{Proofs}
In this section, also denote $m^+(f):=\amax(f)$ and $m^-(f):=\amin(f)$ for brevity.
\begin{proof}[Proof of Theorem \ref{tau-rho-2}]
    Let $C=C_{e,\alpha}\ast C_{e,\beta}$ and $\alpha,\beta\in\mathscr{T}_{\text{PLMS}}$ with partitions $[a_i]_{i=0}^P$ and $[b_j]_{j=0}^Q$, respectively.
    From Theorem \ref{tau-rho-1}, it remains to show that $\tau_{C_{e,\alpha}\ast C_{e,\beta}}=\tau_{C_{e,\alpha}}\tau_{C_{e,\beta}}$.
    Recall from \eqref{composite-1} that if $k=(i-1)Q+j$, then
    \begin{align}
        m^+\paren{(\beta\circ\alpha)_k}&=\begin{cases}
            c_k & \text{if}\ \paren{\beta\circ\alpha}_k\ \text{is increasing};\\
            c_{k-1} & \text{if}\ \paren{\beta\circ\alpha}_k\ \text{is decreasing},
        \end{cases} \nonumber\\
        &=\begin{cases}
            a_{i-1}+m^+\paren{\beta_j}\triangle(a_i) & \text{if}\ \alpha_i\ \text{is increasing};\\
            a_i-m^+\paren{\beta_{Q-j+1}}\triangle(a_i) & \text{if}\ \alpha_i\ \text{is decreasing},
        \end{cases} \nonumber\\
        &=m^-\paren{\alpha_i}+m^+\paren{\widehat{\beta_j}}\brac{m^+\paren{\alpha_i}-m^-\paren{\alpha_i}}\label{alternate-endpoint},
    \end{align}
    where $\widehat{\beta_j}=\begin{cases}
        \beta_j & \text{if}\ \alpha_i\ \text{is increasing};\\
        \beta_{Q-j+1} & \text{if}\ \alpha_i\ \text{is decreasing}.
    \end{cases}$
    Note that when $j$ runs from $1$ to $Q$, $Q-j+1$ runs under the same bound with the opposite direction. 
    Hence by Theorem \ref{tau-rho-1}, we obtain
    \begin{align*}
        \tau_C&=2\sum_{k=1}^{PQ}m^+\paren{(\beta\circ\alpha)_k}\triangle\paren{(\beta\circ\alpha)_k}-1\\
        &=2\sum_{i=1}^{P}m^-(\alpha_i)\triangle(a_i)+2\sum_{i=1}^{P}\sum_{j=1}^{Q}m^+\paren{\beta_j}\brac{m^+\paren{\alpha_i}-m^-\paren{\alpha_i}}\triangle(a_i)\triangle(b_j)-1 \tag{$k=(i-1)Q+j$}\\
        &=2\sum_{i=1}^{P}\paren{a_i^2-a_{i-1}^2}-2\sum_{i=1}^{P}m^+\paren{\alpha_i}\triangle(a_i)+4\sum_{i=1}^{P}\sum_{j=1}^{Q}\brac{m^+\paren{\alpha_i}\cdot m^+\paren{\beta_j}}\triangle(a_i)\triangle(b_j)\\
        &\ \ \ -2\sum_{i=1}^{P}\paren{a_i^2-a_{i-1}^2}\sum_{j=1}^{Q}m^+\paren{\beta_j}\triangle(b_j)-1\\
        &=1-2\sum_{i=1}^{P}m^+\paren{\alpha_i}\triangle(a_i)-2\sum_{j=1}^{Q}m^+\paren{\beta_j}\triangle(b_j)+4\sum_{i=1}^{P}\sum_{j=1}^{Q}\brac{m^+\paren{\alpha_i}\triangle(a_i)}\brac{m^+\paren{\beta_j}\triangle(b_j)}\\
        &=\brac{2\sum_{i=1}^{P}m^+\paren{\alpha_i}\triangle(a_i)-1}\brac{2\sum_{j=1}^{Q}m^+\paren{\beta_j}\triangle(b_j)-1}=\tau_{C_{e,\alpha}}\cdot\tau_{C_{e,\beta}}. \qedhere
    \end{align*}
\end{proof}
\begin{proof}[Proof of Theorem \ref{tau-rho-3}]
    To find $\tau_C=1-4\iint\frac{\partial C}{\partial x}\frac{\partial C}{\partial y}dxdy$, from \eqref{PLMS-fac}, in the case that $\alpha_i$ and $\beta_j$ have the same monotonicity,
    \begin{align*}
        \iint_{R_{ij}}\frac{\partial C}{\partial x}\frac{\partial C}{\partial y}dxdy&=a_{i-1}b_j\times\text{Area}\Bigg(\begin{tikzpicture}[baseline=0.6\baselineskip]
        \draw(0,0)--(0.6,0.8)--(0,0.8)--cycle;
        \node[left] at (0,0.4){\footnotesize$\triangle(b_j)$};
        \node[above] at (0.3,0.8){\footnotesize$\triangle(a_i)$};
    \end{tikzpicture}\Bigg)+a_ib_{j-1}\times\text{Area}\Bigg(\begin{tikzpicture}[baseline=0.4\baselineskip]
        \draw(0,0)--(0.6,0)--(0.6,0.8)--cycle;
        \node[right] at (0.6,0.4){\footnotesize$\triangle(b_j)$};
        \node[below] at (0.3,0){\footnotesize$\triangle(a_i)$};
    \end{tikzpicture}\Bigg)\\
    &=\frac{\triangle(a_i)\triangle(b_j)}{2}\left[a_{i-1}b_j+a_ib_{j-1}\right]\\
    &=\frac{\triangle(a_i)\triangle(b_j)}{2}\left[m^-(\alpha_i)\cdot m^+(\beta_j)+m^+(\alpha_i)\cdot m^-(\beta_j)\right].
    \end{align*}
    In the case where $\alpha_i$ and $\beta_j$ have different monotonicities, $\iint_{R_{ij}}\frac{\partial C}{\partial x}\frac{\partial C}{\partial y}dxdy$ also has the same form as the previous case.
    Hence,
    \begin{align*}
        \tau_C&=1-2\sum_{i=1}^{P}\sum_{j=1}^{Q}\triangle(a_i)\triangle(b_j)\left[m^-(\alpha_i)\cdot m^+(\beta_j)+m^+(\alpha_i)\cdot m^-(\beta_j)\right] \tag{$\ast$}\\
        &=1+4\sum_{i=1}^{P}\sum_{j=1}^{Q}\brac{m^+(\alpha_i)\cdot m^+(\beta_j)}\triangle(a_i)\triangle(b_j)\\
        &\ \ \ -2\sum_{i=1}^{P}\sum_{j=1}^{Q}m^+(\beta_j)\triangle(b_j)\paren{a_i^2-a_{i-1}^2}-2\sum_{i=1}^{P}\sum_{j=1}^{Q}m^+(\alpha_i)\triangle(a_i)\paren{b_j^2-b_{j-1}^2}\\
        &=1-2\sum_{i=1}^{P}m^+(\alpha_i)\triangle(a_i)-2\sum_{j=1}^{Q}m^+(\beta_j)\triangle(b_j)+4\sum_{i=1}^{P}\sum_{j=1}^{Q}\brac{m^+(\alpha_i)\triangle(a_i)}\brac{m^+(\beta_j)\triangle(b_j)}\\
        &=\left[2\sum_{i=1}^{P}m^+(\alpha_i)\triangle(a_i)-1\right]\left[2\sum_{j=1}^{Q}m^+(\beta_j)\triangle(b_j)-1\right]=\tau_{C_{e,\alpha}}\cdot\tau_{C_{\beta,e}}.
    \end{align*}
    To find $\rho_C$, consider the case where $\alpha_i$ and $\beta_j$ have the same monotonicity,
    \begin{align*}
        \iint_{R_{ij}}C(x,y)dxdy&=\int_{b_{j-1}}^{b_j}\int_{a_{i-1}}^{\left(\alpha_i^{-1}\circ\beta_j\right)(y)}\left[b_j(x-a_{i-1})+a_{i-1}y\right]dxdy\\
        &\ \ \ +\int_{a_{i-1}}^{a_i}\int_{b_{j-1}}^{\left(\beta_j^{-1}\circ\alpha_i\right)(x)}\left[b_{j-1}x+a_i(y-b_{j-1})\right]dydx\\
        &=\frac{1}{6}\triangle(a_i)\triangle(b_j)\left[a_{i-1}(b_j+b_{j-1})+a_ib_j\right]\\
        &\ \ \ +\frac{1}{6}\triangle(a_i)\triangle(b_j)\left[a_{i-1}b_{j-1}+a_i(b_j+b_{j-1})\right]\\
        &=\frac{1}{6}\triangle(a_i)\triangle(b_j)\left[2(a_i+a_{i-1})(b_j+b_{j-1})-a_{i-1}b_j-a_{i}b_{j-1}\right]\\
        &=\frac{1}{3}\paren{a_i^2-a_{i-1}^2}\paren{b_j^2-b_{j-1}^2}\\
        &\ \ -\frac{1}{6}\triangle(a_i)\triangle(b_j)\left[m^-(\alpha_i)\cdot m^+(\beta_j)+m^+(\alpha_i)\cdot m^-(\beta_j)\right].
    \end{align*}
    In the case where $\alpha_i$ and $\beta_j$ have different monotonicities, $\iint_{R_{ij}}C(x,y)dxdy$ has the same form as the previous case.
    Hence,
    \begin{align*}
        \rho_C
        &=4\sum_{i=1}^P\sum_{j=1}^Q(a_i^2-a_{i-1}^2)(b_j^2-b_{j-1}^2)-3\\
        &\ \ -2\sum_{i=1}^{P}\sum_{j=1}^{Q}\triangle(a_i)\triangle(b_j)\left[m^-(\alpha_i)\cdot m^+(\beta_j)+m^+(\alpha_i)\cdot m^-(\beta_j)\right]\\
        &=1-2\sum_{i=1}^{P}\sum_{j=1}^{Q}\triangle(a_i)\triangle(b_j)\left[m^-(\alpha_i)\cdot m^+(\beta_j)+m^+(\alpha_i)\cdot m^-(\beta_j)\right],
    \end{align*}
    which is the same as $(\ast)$ above, and therefore, $\rho_C=\tau_C=\tau_{C_{e,\alpha}}\tau_{C_{\beta,e}}=\rho_{C_{e,\alpha}}\rho_{C_{\beta,e}}$ where the last equality follows from Theorem \ref{tau-rho-1}.
\end{proof}
\begin{proof}[Proof of Theorem \ref{thm7:18Dec}]
    First of all, we claim that
    \begin{equation}
        1-2\sum_{i=1}^{P}\sum_{j=1}^{Q}(a_{i-1}b_j+a_ib_{j-1})\triangle(a_i)\triangle(b_j)=\sum_{i=1}^{P}\sum_{j=1}^{Q}(\triangle(a_i))^2(\triangle(b_j))^2.\tag{$\ast\ast$}
    \end{equation}
    Notice that the expression on the left-hand side is
    \begin{align*}
        &=1-2\sum_{i=1}^{P}\sum_{j=1}^{Q}\brac{a_ib_j+a_{i-1}b_{j-1}-\triangle(a_i)\triangle(b_j)}\triangle(a_i)\triangle(b_j)\\
        &=1-\sum_{i=1}^{P}\sum_{j=1}^{Q}\brac{2a_ib_j+2a_{i-1}b_{j-1}-\triangle(a_i)\triangle(b_j)}\triangle(a_i)\triangle(b_j)+\sum_{i=1}^{P}\sum_{j=1}^{Q}(\triangle(a_i))^2(\triangle(b_j))^2\\
        &=1-\sum_{i=1}^{P}(a_i^2-a_{i-1}^2)\sum_{j=1}^{Q}(b_j^2-b_{j-1}^2)+\sum_{i=1}^{P}\sum_{j=1}^{Q}(\triangle(a_i))^2(\triangle(b_j))^2\\
        &=\sum_{i=1}^{P}\sum_{j=1}^{Q}(\triangle(a_i))^2(\triangle(b_j))^2.
    \end{align*}
    We have by Lemma \ref{lem6:18Dec} that for any $k\in\left\{1,2,\dots,P\right\}$ and $\ell\in\left\{1,2,\dots,Q\right\}$,
    \begin{align*}
        \iint_{R_{k\ell}}A(u,v)dudv&=\iint_{R_{k\ell}}\left[a_{k-1}v+b_{\ell-1}u-a_{k-1}b_{\ell-1}+\triangle(a_k)\triangle(b_{\ell})(C_k\ast D_{\ell})\left(x(u),y(v)\right)\right]dvdu \tag*{$\bigg(x(u)=\dfrac{u-a_{k-1}}{\triangle(a_k)},\,y(v)=\dfrac{v-b_{\ell-1}}{\triangle(b_{\ell})}\bigg)$}\\
        &=\frac{1}{2}\triangle(a_k)\triangle(b_{\ell})\brac{a_{k-1}(b_{\ell}+b_{\ell-1})+b_{\ell-1}\paren{a_k+a_{k-1}}-2a_{k-1}b_{\ell-1}}\\
        &\ \ \ +(\triangle(a_k))^2(\triangle(b_{\ell}))^2\iint_{\I^2}(C_k\ast D_{\ell})(x,y)dydx\\
        &=\frac{1}{2}\brac{a_{k-1}b_{\ell}+a_kb_{\ell-1}}\triangle(a_k)\triangle(b_{\ell})+(\triangle(a_k))^2(\triangle(b_{\ell}))^2\iint_{\I^2}(C_k\ast D_{\ell})(x,y)dydx.
    \end{align*}
    Since $\displaystyle\rho_A=12\iint_{\I^2}A(u,v)dudv-3$,
    \begin{align*}
        \rho_A&=6\sum_{i=1}^{P}\sum_{j=1}^{Q}\brac{a_{i-1}b_{j}+a_ib_{j-1}}\triangle(a_i)\triangle(b_j)+12\sum_{i=1}^{P}\sum_{j=1}^{Q}(\triangle(a_i))^2(\triangle(b_j))^2\iint_{\I^2}(C_i\ast D_j)(x,y)dxdy-3\\
        &=12\sum_{i=1}^{P}\sum_{j=1}^{Q}(\triangle(a_i))^2(\triangle(b_j))^2\iint_{\I^2}(C_i\ast D_j)(x,y)dxdy-3\sum_{i=1}^{P}\sum_{j=1}^{Q}(\triangle(a_i))^2(\triangle(b_j))^2\tag{by $(\ast\ast)$}\\
        &=\sum_{i=1}^{P}\sum_{j=1}^{Q}(\triangle(a_i))^2(\triangle(b_j))^2\left[12\iint_{\I^2}(C_i\ast D_j)(x,y)dxdy-3\right]\\
        &=\sum_{i=1}^{P}\sum_{j=1}^{Q}(\triangle(a_i))^2(\triangle(b_j))^2\rho_{C_i\ast D_j}.
    \end{align*}
    To calculate $\displaystyle\tau_A=1-4\iint_{\I^2}\frac{\partial A}{\partial u}\frac{\partial A}{\partial v}dudv$, we first derive $\dfrac{\partial A}{\partial u}(u,v)$ and $\dfrac{\partial A}{\partial v}(u,v)$ for any $(u,v)\in R_{k\ell}$.
    \begin{align*}
        \frac{\partial A}{\partial u}(u,v)&=b_{\ell-1}+\triangle(b_{\ell})\partial_1(C_k\ast D_{\ell})\left(\frac{u-a_{k-1}}{\triangle(a_k)},\frac{v-b_{\ell-1}}
        {\triangle(b_{\ell})}\right),\\
        \frac{\partial A}{\partial v}(u,v)&=a_{k-1}+\triangle(a_{k})\partial_2(C_k\ast D_{\ell})\left(\frac{u-a_{k-1}}{\triangle(a_k)},\frac{v-b_{\ell-1}}
        {\triangle(b_{\ell})}\right).
    \end{align*}
    Setting $x(u)=\dfrac{u-a_{k-1}}{\triangle(a_k)}$ and $y(v)=\dfrac{v-b_{\ell-1}}{\triangle(b_{\ell})}$. This implies
    \begin{align*}
        \iint_{R_{k\ell}}\frac{\partial A}{\partial u}\frac{\partial A}{\partial v}dudv&=\iint_{R_{k\ell}}\left[a_{k-1}b_{\ell-1}+b_{\ell-1}\triangle(a_k)\partial_2(C_k\ast D_{\ell})\left(x(u),y(v)\right)\right]dudv\\
        &\ \ \ +\iint_{R_{k\ell}}\left[a_{k-1}\triangle(b_{\ell})\partial_1(C_k\ast D_{\ell})\left(x(u),y(v)\right)\right]dudv\\
        &\ \ \ +\iint_{R_{k\ell}}\left[\triangle(a_k)\triangle(b_{\ell})\partial_1(C_k\ast D_{\ell})\left(x(u),y(v)\right)\partial_2(C_k\ast D_{\ell})\left(x(u),y(v)\right)\right]dudv\\
        &=a_{k-1}b_{\ell-1}\triangle(a_k)\triangle(b_{\ell})+\frac{1}{2}b_{\ell-1}(\triangle(a_k))^2\triangle(b_{\ell})+\frac{1}{2}a_{k-1}\triangle(a_k)(\triangle(b_{\ell}))^2\\
        &\ \ \ +(\triangle(a_k))^2(\triangle(b_{\ell}))^2\iint_{\I^2}\partial_1(C_k\ast D_{\ell})(x,y)\partial_2(C_k\ast D_{\ell})(x,y)dxdy,\\
        \intertext{of which the first three terms can be written as}
        &=\triangle(a_k)\triangle(b_{\ell})\left[a_{k-1}b_{\ell-1}+\frac{1}{2}b_{\ell-1}\triangle(a_k)+\frac{1}{2}a_{k-1}\triangle(b_{\ell})\right]\\
        &=\frac{1}{2}\brac{a_kb_{\ell-1}+a_{k-1}b_{\ell}}\triangle(a_k)\triangle(b_{\ell})
    \end{align*}
    Therefore, by $(\ast\ast)$,
    \begin{align*}
        \tau_A&=1-2\sum_{i=1}^{P}\sum_{j=1}^{Q}\brac{a_ib_{j-1}+a_{i-1}b_j}\triangle(a_i)\triangle(b_j)\\
        &\ \ \ -4\sum_{i=1}^{P}\sum_{j=1}^{Q}(\triangle(a_i))^2(\triangle(b_j))^2\iint_{\I^2}\partial_1(C_i\ast D_j)\partial_2(C_i\ast D_j)dxdy\\
        &=\sum_{i=1}^{P}\sum_{j=1}^{Q}(\triangle(a_i))^2(\triangle(b_j))^2-4\sum_{i=1}^{P}\sum_{j=1}^{Q}(\triangle(a_i))^2(\triangle(b_j))^2\iint_{\I^2}\partial_1(C_i\ast D_j)\partial_2(C_i\ast D_j)dxdy\\
        &=\sum_{i=1}^{P}\sum_{j=1}^{Q}(\triangle(a_i))^2(\triangle(b_j))^2\left[1-4\iint_{\I^2}\partial_1(C_i\ast D_j)\partial_2(C_i\ast D_j)dxdy\right]\\
        &=\sum_{i=1}^{P}\sum_{j=1}^{Q}(\triangle(a_i))^2(\triangle(b_j))^2\tau_{C_i\ast D_j}. \qedhere
    \end{align*}
\end{proof}
\begin{proof}[Proof of Theorem \ref{chat-2}]
    Let $C=C_{\beta,e}\ast C_{\alpha,e}$ and $\alpha,\beta\in\mathscr{T}_{\text{PLMS}}$ with the corresponding partitions $[a_i]_{i=0}^P$ and $[b_j]_{j=0}^Q$, respectively.
    Recall \eqref{alternate-endpoint} that if $k=(i-1)Q+j$, then $m^+\paren{(\beta\circ\alpha)_k}=m^-\paren{\alpha_i}+m^+\paren{\widehat{\beta_j}}\brac{m^+\paren{\alpha_i}-m^-\paren{\alpha_i}}$, and similarly, $m^+\paren{(\beta\circ\alpha)_k}=m^-\paren{\alpha_i}+m^-\paren{\widehat{\beta_j}}\brac{m^+\paren{\alpha_i}-m^-\paren{\alpha_i}}$.
    That is, 
    \begin{align*}
        c_kc_{k-1}&=m^+\paren{(\beta\circ\alpha)_k}m^-\paren{(\beta\circ\alpha)_k}\\
        &=\paren{m^-\paren{\alpha_i}}^2+m^-\paren{\alpha_i}\brac{m^+\paren{\alpha_i}-m^-\paren{\alpha_i}}\brac{\widehat{b_j}+\widehat{b_{j-1}}}+\widehat{b_j}\widehat{b_{j-1}}\brac{m^+\paren{\alpha_i}-m^-\paren{\alpha_i}}^2\\
        &=\paren{m^-\paren{\alpha_i}}^2+m^-\paren{\alpha_i}\brac{m^+\paren{\alpha_i}-m^-\paren{\alpha_i}}\brac{\widehat{b_j}+\widehat{b_{j-1}}}+\widehat{b_j}\widehat{b_{j-1}}\brac{a_i-a_{i-1}}^2.
    \end{align*} 
    Thus, by Theorem \ref{chat-1}, we obtain
    \begin{align*}
        \xi_C&=1-3\sum_{k=1}^{PQ}c_kc_{k-1}\triangle(c_k)\\
        &=1-3\sum_{i=1}^{P}\paren{m^-\paren{\alpha_i}}^2\triangle(a_i)-3\sum_{i=1}^{P}m^-\paren{\alpha_i}\brac{m^+\paren{\alpha_i}-m^-\paren{\alpha_i}}\triangle(a_i)\sum_{j=1}^{Q}\paren{b_j^2-b_{j-1}^2}\\
        &\ \ \ -3\sum_{i=1}^{P}\sum_{j=1}^{Q}\brac{a_i-a_{i-1}}^3b_jb_{j-1}\triangle(b_j) \tag{$k=(i-1)Q+j$}\\
        &=1-3\sum_{i=1}^Pa_ia_{i-1}\triangle(a_i)-3\sum_{i=1}^{P}\paren{a_i^3-a_{i-1}^3}\sum_{j=1}^{Q}b_jb_{j-1}\triangle(b_j)+9\sum_{i=1}^{P}\brac{a_ia_{i-1}\triangle(a_i)}\brac{b_jb_{j-1}\triangle(b_j)}\\
        &=\brac{1-3\sum_{j=1}^{Q}b_jb_{j-1}\triangle(b_j)}\brac{1-3\sum_{i=1}^{P}a_ia_{i-1}\triangle(a_i)}=\xi_{C_{\beta,e}}\cdot\xi_{C_{\alpha,e}}.\qedhere
    \end{align*}
\end{proof}
\subsection{R-code for simulation}
    We can simulate an implicit dependence copula of the form $C_{e,f}\ast_{\mathcal{A}} C_{g,e}$, where $f,g\in\mathscr{T}_{\text{PLMS}}$, $\mathcal{A}=\paren{A_t}_{t\in\I}\subseteq\mathscr{C}$ such that $A_t=C_{\ell}$ for all $t\in\paren{t_{\ell-1},t_{\ell}}$, and $[t_{\ell}]_{\ell=0}^{L}$ is a partition for $[0,1]$, using the following \texttt{R}-code.
    \begin{lstlisting}
#################################################################
# Basic functions
#################################################################
# define f
rdef <- function(x,Par,Signs){
  j <- rank(c(Par,x))[length(Par)+1]-1
  y <- (x-Par[j])/(Par[j+1]-Par[j])
  if (Signs[j]==-1) {y <- 1-y}
  return(y)
}
# define f^{-1}
rdef_inv <- function(z,Par,Signs,pb){
  which_i <- sample.int(length(Par)-1,size=1,replace=TRUE,prob=pb)
  if (Signs[which_i]==1)
  {x <- Par[which_i]+z*(Par[which_i+1]-Par[which_i])}
  else
  {x <- Par[which_i+1]-z*(Par[which_i+1]-Par[which_i])}
  return(x)
}

#################################################################
# Generate IDC
#################################################################
# generate related transformation matrix
cop_tran_matrix <- function(func,Par1,Par2) {
  cop_matrix <- matrix(0,nrow=length(Par1)-1,ncol=length(Par2)-1)
  for (i in (1:length(Par1)-1)) {
    for (j in (1:length(Par2)-1)) {
      cop_matrix[i,j] <- func(Par1[i+1],Par2[j+1])-func(Par1[i],Par2[j+1])-func(Par1[i+1],Par2[j])+func(Par1[i],Par2[j])
    }
  }
  return(cop_matrix)
}
# collect transformation matrices
collect_matrix <- function(Par1,Par2,join_cop) {
  card_cop <- length(join_cop)
  collect_join <- list()
  for (i in 1:card_cop) {
    collect_join <- append(collect_join,list(cop_tran_matrix(join_cop[[i]],Par1,Par2)))
  }
  return(collect_join)
}
# generate copula
simuproduct <- function(n,Par1,Sgn1,Par2,Sgn2,range_sep,join_cop){
  trans_collect <- collect_matrix(Par1,Par2,join_cop)
  C <- matrix(0,nrow=n,ncol=2)
  C[,1] <- runif(n)
  for (i in (1:n))
  {
    j <- rank(c(Par1,C[i,1]))[length(Par1)+1]-1
    temp <- rdef(C[i,1],Par1,Sgn1)
    k <- rank(c(range_sep,temp))[length(range_sep)+1]-1
    tran_matrix <- trans_collect[[k]]
    prob_dist <- tran_matrix[j,]
    C[i,2] <- rdef_inv(temp,Par2,Sgn2,prob_dist)
  }
  return(C)
}

#################################################################
# Basic info for f, g, and joining copulas (random)
#################################################################
max_no_int <- 5
partition_fg <- floor(max_no_int*runif(2))
Pi1 <- sort(runif(partition_fg[1]))
Pi1_help <- c(0,Pi1,1)
Sign1 <- sample(c(-1,1),partition_fg[1]+1,replace=TRUE)
Pi2 <- sort(runif(partition_fg[2]))
Pi2_help <- c(0,Pi2,1)
Sign2 <- sample(c(-1,1),partition_fg[2]+1,replace=TRUE)
no_range <- 3
range_sep <- sort(runif(no_range-1))
range_sep_help <- c(0,range_sep,1)
cop_list <- c(function(u,v) {max(u+v-1,0)},
              function(u,v) {min(u,v)},
              function(u,v) {u*v})
join_cop <- sample(cop_list,no_range,replace=TRUE)

#################################################################
# Find product and plot all
#################################################################
# simulate data
n <- 5000
simdata <- simuproduct(n,Pi1_help,Sign1,Pi2_help,Sign2,range_sep_help,join_cop)
# plot data
par(mfrow=c(1,1),pty="s")
plot(simdata[,1],simdata[,2],col="violet",main="Generalized product of C_{ef} and C_{ge}",pch=16,cex=0.5,xlab="x",ylab="y",asp=1,xlim=c(0,1),ylim=c(0,1))
abline(v=Pi1_help,lty="dashed",col="red")
abline(h=Pi2_help,lty="dashed",col="red")
par(mfrow=c(1,2))
hist(simdata[,1],main="histogram for x",xlab="x")
hist(simdata[,2],main="histogram for y",xlab="y")
\end{lstlisting}
This \texttt{R}-script consists of four parts.
\begin{itemize}
    \item The basic functions \texttt{rdef} and \texttt{rdef\_inv}, where \texttt{rdef} sends $x\mapsto h(x)$ and \texttt{rdef\_inv} sends $h(x)\mapsto x$ (\texttt{rdef\_inv} is actually not a function, so it has a parameter (pb) to randomize an element from the inverse image of $h(x)$) \cite{Mai17}.
    \item The functions to generate an implicit dependence copula: 
    \begin{itemize}
        \item \texttt{cop\_tran\_matrix} generates a transformation matrix from the given joining copula func and partitions of $[0,1]$ on both axes.
        \item \texttt{collect\_matrix} generates a collection of transformation matrices where the joining copulas in \texttt{join\_cop vary}.
        \item \texttt{simuproduct} generates 2-dimensional vectors, where the first coordinate ($x$) is randomized from $\mathcal{U}[0,1]$ and the second one is $g^{-1}(f(x))$, where the functions \texttt{rdef} and \texttt{rdef\_inv} are used to define $f$ and $g^{-1}$, respectively \cite{Mai17}.
        Here, the chosen $y$-value from a given $x$ follows from the value $f(x)$ and the transformation matrix corresponding to the interval covering $f(x)$.
    \end{itemize}
    \item Random part, which determines partitions (Pi) and monotonicity (Sign) of functions $f$ and $g$ (it can be deterministic).
    This part also indicates ranges as a partition of $[0,1]$ (\texttt{range\_sep}) and joining copulas (\texttt{join\_cop}) used to define transformation matrices.
    \item Simulation part, which determines the number of simulations (n) and simulates a matrix of 2-dimensional vectors.
    This matrix is used to plot the support of $C_{e,f}\ast_A C_{g,e}$, as well as the histograms for both coordinates.
\end{itemize}
The next script compares the measures from the simulated implicit dependence copula.
\begin{lstlisting}
# tau & rho for PLMS
tautheory <- function(Pi,Signs)
{
  tau_help <- 0
  Pi_help  <- c(0,Pi,1)
  for (i in (1:(length(Pi)+1)))
  {
    j <- 1
    if (Signs[i]==1)
    {tau_help <- tau_help+Pi_help[i+1]*(Pi_help[i+1]-Pi_help[i])}
    else
    {tau_help <- tau_help+Pi_help[i]*(Pi_help[i+1]-Pi_help[i])}
  }
  tau_help <- 2*tau_help-1
  return(tau_help)
}
# Chatterjee's for PLMS
XiTheory <- function(part) {
  xi_help <- 0
  part_help <- c(0,part,1)
  for (i in (1:(length(part_help)-1))) {
    xi_help <- xi_help+(part_help[i+1]-part_help[i])*part_help[i+1]*part_help[i]
  }
  xi_help <- 1-3*xi_help
  return(xi_help)
}
# calculate from data (theoretical)
tautheory(Pi1,Sign1)
tautheory(Pi2,Sign2)
tautheory(Pi1,Sign1)*tautheory(Pi2,Sign2)
Xi_UV <- XiTheory(Pi2)
Xi_VU <- XiTheory(Pi1)
# calculate from data (actual)
library(pcaPP)
library(XICOR)
cor.fk(simdata)[2,1]
cor(simdata,method = "spearman")[2,1]
calculateXI(simdata[,1],simdata[,2])
calculateXI(simdata[,2],simdata[,1])
\end{lstlisting}
This \texttt{R}-script consists of 3 parts.
\begin{itemize}
    \item Set up the theoretical Kendall's tau (=Spearman's rho) and Chatterjee's correlation for $C_{e,\alpha}$ and $C_{\alpha,e}$ ($\alpha\in\mathscr{T}_{\text{PLMS}}$), respectively, obtained from Theorems~\ref{tau-rho-1} and \ref{chat-1}.
    \item Calculate Kendall's tau for $C_{e,f},C_{g,e},C_{e,f}\ast C_{g,e}$ and Chatterjee's correlations for $C_{e,f}\ast C_{g,e}$ and its transpose from the set up functions in the first part, as well as Theorems~\ref{tau-rho-3} and \ref{chat-3}.
    \item Calculate the empirical Kendall's tau, Spearman's rho, and Chatterjee's coefficients of $C_{e,f}\ast_AC_{g,e}$ from the simulated values.
    Here, the functions \texttt{cor.fk} (from \texttt{pcaPP} package) and \texttt{calculateXI} (from \texttt{XICOR} package) are used to compute $\tau_C$ and $\xi_C$, respectively.
\end{itemize}
\section*{Acknowledgements}
This research project is supported by the Second Century Fund (C2F), Chulalongkorn University.			
The first author also thanks the Chair of Mathematical Finance, Technical University of Munich (TUM), for the opportunity to work as a visiting researcher, including provision of office space and access to research resources during the author's stay at TUM.

\end{document}